\documentclass[12pt]{amsart}
\usepackage{amssymb,amsmath}
\usepackage{geometry}    
\usepackage{mathrsfs}

\usepackage{vmargin}

\usepackage{colortbl}  % farben!!!!!!!!!!!

\setmargrb{1in}{1in}{1in}{1in}

\newtheorem{theorem}{Theorem}[section]
\newtheorem{lemma}[theorem]{Lemma}
\newtheorem{proposition}[theorem]{Proposition}
\newtheorem{corollary}[theorem]{Corollary}
\theoremstyle{definition}

\newtheorem{remark}{Remark}
\newtheorem{example}[theorem]{Example}
\newtheorem{conjecture}[theorem]{Conjecture}

\newcommand{\scp}[1]{\langle#1\rangle}

\begin{document}

\title{Going-up theorems for simultaneous Diophantine approximation}

\author{Johannes Schleischitz} 

\thanks{Middle East Technical University, Northern Cyprus Campus, Kalkanli, G\"uzelyurt \\
	jschleischitz@outlook.com, johannes@metu.edu.tr}

%\vspace{8mm}

\begin{abstract}
We establish several new inequalities linking classical exponents
of Diophantine approximation associated to a real vector 
$\underline{\xi}=(\xi,\xi^{2},\ldots,\xi^{N})$, in
various dimensions $N$.
We thereby obtain variants, and partly refinements, of recent results 
of Badziahin and Bugeaud. We further implicitly 
recover inequalities of Bugeaud and Laurent
as special cases, with new proofs. Similar estimates concerning
general real vectors (not on the Veronese curve)
with $\mathbb{Q}$-linearly independent coordinates are addressed as well.
Our method is based on Minkowski's Second Convex Body Theorem,
applied in the framework of parametric geometry of numbers introduced by
Schmidt and Summerer. We also frequently employ
Mahler's Duality result on polar convex bodies.
\end{abstract}

\maketitle

{\footnotesize{

Math Subject Classification 2010: 11J13, 11J83}}

\vspace{4mm}

\section{Introduction and outline}  \label{introd}

Let $N\geq 1$ be an integer and
$\underline{\xi}=(\xi_{1},\ldots,\xi_{N})\in\mathbb{R}^{N}$. 
We denote 
by  $\lambda_{N}(\underline{\xi})$ the ordinary exponent 
of simultaneous approximation, 
defined as the supremum of real $\lambda$ such that
\begin{equation}  \label{eq:lambdar}
1\leq \vert x\vert \leq X, \qquad\qquad \max_{1\leq j\leq N} \vert \xi_{j}x-y_{j}\vert \leq X^{-\lambda},  
\end{equation}
has a solution $(x,y_{1},\ldots,y_{N})\in{\mathbb{Z}^{N+1}}$
for arbitrarily large values of $X$. 
Let the ordinary exponent of linear form approximation
$w_{N}(\underline{\xi})$ be the supremum 
of real $w$ such that 
\begin{equation} \label{eq:lammda}
\max_{1\leq j\leq N} \vert x_{j}\vert\leq X, \qquad\qquad 
\vert x_{0}+\xi_{1}x_{1}+\cdots +\xi_{N}x_{N}\vert \leq X^{-w}
\end{equation}
has a solution 
$(x_{0},\ldots,x_{N})\in\mathbb{Z}^{N+1}$ 
for arbitrarily large $X$. Similarly, let
the uniform exponents $\widehat{\lambda}_{N}(\underline{\xi})$ and 
$\widehat{w}_{N}(\underline{\xi})$ respectively be
given as the respective suprema such that \eqref{eq:lambdar}
and \eqref{eq:lammda} have a solutions for {\em all} large $X$. 
Dirichlet's Theorem implies for any 
$\underline{\xi}\in\mathbb{R}^{N}$
\begin{equation}  \label{eq:dirichlet}
\lambda_{N}(\underline{\xi})\geq \widehat{\lambda}_{N}(\underline{\xi})\geq \frac{1}{N}, \qquad\qquad
w_{N}(\underline{\xi})\geq 
\widehat{w}_{N}(\underline{\xi})\geq N.
\end{equation}
In this paper, we are mostly concerned with
the special case $\underline{\xi}=(\xi,\xi^{2},\ldots,\xi^{N})$ for $\xi\in\mathbb{R}$, that is points on a Veronese curve. 
We then denote the 
exponents $\lambda_{N}(\underline{\xi}),w_{N}(\underline{\xi})$
simply by $\lambda_{N}(\xi), w_{N}(\xi)$ respectively, and likewise
the respective uniform exponents by $\widehat{\lambda}_{N}(\xi),\widehat{w}_{N}(\xi)$. 
\footnote{We believe
that this slight abuse of notation will 
improve readability of this paper, but want to remark 
that other notions for the exponents with respect to 
general $\underline{\xi}$ in $\mathbb{R}^N$, like $\omega_{N}(\underline{\xi}), \widehat{\omega}_{N}(\underline{\xi})$ and $\omega^{\ast}_{N}(\underline{\xi}), \widehat{\omega}_{N}^{\ast}(\underline{\xi})$, are more common.
The exponents on the Veronese curve are denoted
as in the standard literature.}
Thereby, we see that any real $\xi$ gives rise to four sequences of exponents
\begin{equation} \label{eq:joint}
(\lambda_{N}(\xi))_{N\geq 1}, \qquad (\widehat{\lambda}_{N}(\xi))_{N\geq 1},
\qquad (w_{N}(\xi))_{N\geq 1}, \qquad (\widehat{w}_{N}(\xi))_{N\geq 1}.
\end{equation}
Clearly the exponents
$\lambda_{N}(\xi), \widehat{\lambda}_{N}(\xi)$ are non-increasing with $N$ 
whereas the exponents $w_{N}(\xi), \widehat{w}_{N}(\xi)$ form
non-decreasing sequences. Ordinary exponents may take the value $+\infty$, whereas uniform exponents turn out to be always less than twice the trivial lower bounds in \eqref{eq:dirichlet}, see Remarks~\ref{reh}, \ref{reh4} below for refinements. Only for $N=2$ numbers satisfying $\widehat{\lambda}_{N}(\xi)>1/N$ or $\widehat{w}_{N}(\xi)>N$
have been found, see Roy~\cite{roy1},~\cite{roy4},  Fischler~\cite{fischler}, Bugeaud, Laurent~\cite{buglau} and Poels~\cite{poels}.

Investigation of 
these exponents with emphasis on Veronese curves is partly motivated by
well-known connections to the problem of approximation to real numbers by algebraic numbers (integers) related to Wirsing's problem, 
see Wirsing~\cite{wirsing}, Davenport, Schmidt~\cite{davsh} and Badziahin, Schleischitz~\cite{bswirsing}.
The main purpose of this paper is to establish
new inequalities interconnecting these exponents, in various dimensions.
Concretely,
in Sections~\ref{se2},~\ref{mix} we establish
several lower bounds for $\lambda_{k}(\xi)$ in terms of various exponents
of index $n\leq k$, and compare them.
Thereby, we complement a recent paper by Badziahin and Bugeaud~\cite{badbug},
as well as previous work of the author, 
especially~\cite{schlei}, \cite{ichann}. 
As a byproduct we further find new proofs of transference inequalities by Bugeaud, Laurent~\cite{bl2010}. Section~\ref{qli} treats analogous topics
for general $\mathbb{Q}$-linearly independent vectors. Estimates are
naturally weaker here and it is included rather for sake of completeness
and to motivate a comprehensive conjecture.
In Section~\ref{paramet} we introduce parametric geometry of numbers,
a key tool in the proofs of the main results carried out in Sections~\ref{6},~\ref{7}.
Finally in Section~\ref{pqli} we provide short proofs 
of the results from Section~\ref{qli}.

\section{Relations between exponents of simultaneous approximation } \label{se2}

\subsection{Going-up Theorems for the sequence $(\lambda_{N}(\xi))_{N\geq 1}$ }

We want to understand relations between the exponents
$\lambda_{N}(\xi)$ associated to real $\xi$ in various dimensions $N$,
thereby to draw information on the {\em joint spectrum} of 
the first sequence in \eqref{eq:joint}, i.e.
all possible sequences $(\lambda_{1}(\xi),\lambda_{2}(\xi),\ldots)$ 
induced by transcendental real $\xi$.
Bugeaud~\cite{bug} was the first to study this topic in detail. 
Among other results, he established the inequalities
\begin{equation} \label{eq:herz}
\lambda_{nk}(\xi)\geq \frac{\lambda_{k}(\xi)-n+1}{n},
\end{equation}
valid for positive integers $k,n$ and any real number $\xi$.
A generalization of \eqref{eq:herz} conjectured
by the author in~\cite{schlei} was proved by Badziahin and Bugeaud~\cite{badbug}.

\begin{theorem}[Badziahin, Bugeaud] \label{juppy}
	For any real number $\xi$ and integers $k\geq n\geq 1$ we have the estimate 
	\begin{equation}  \label{eq:jippy}
	\lambda_{k}(\xi)\geq \frac{n \lambda_{n}(\xi)+n-k}{k}.
	\end{equation}
\end{theorem}

In the special case $\lambda_{n}(\xi)>1$ it had been known before, 
and if even $\lambda_{k}(\xi)>1$ then there is in fact equality, see~\cite[Corollary~1.10]{schlei}. In particular the estimate
is sharp in certain cases. It is tempting to believe that it 
is best possible for all reasonable parameters (i.e. if the 
bound becomes at least $1/k$). If the bound in \eqref{eq:jippy}
is less than $1$, 
this leaves some freedom for $\lambda_{k}(\xi)$. However,
when considering all large $k$ simultaneously,
stringent restrictions
on the joint spectrum were given in~\cite{ichann}.

%We remark that
%it suffices to ask for \eqref{eq:jippy} to be satisfied only for 
%$k=l+1$. This is
%a consequence of the fact that if 
%we let $\tau_{k,l}(x)=(lx+l-k)/k$, 
%then we readily verify $\tau_{m,n}\circ \tau_{n,l}=\tau_{m,l}$
%for all $m\geq n\geq l$.
%Applying the inequality for successive numbers 
%iteratively with the smaller integer ranging from 
%$l$ to $k-1$ and using this
%concetanation identity, we see that the inequality holds
%for any pairs $k,l$.

We refine Theorem~\ref{juppy} by means of introducing uniform
exponents. We further include an alternative bound that is sometimes
stronger.

\begin{theorem} \label{gutersatze0}
	Let $k\geq n\geq 1$ be integers. For any real $\xi$ we have
	\begin{equation} \label{eq:folgerunge0}
	\lambda_{k}(\xi) \geq \frac{(n-1)\lambda_{n}(\xi)+(k-n)\widehat{\lambda}_{n}(\xi)+n-k}{(n-k)\widehat{\lambda}_{n}(\xi)+k-1}.
	\end{equation}
	Moreover, we have
	\begin{equation}  \label{eq:wanndenn}
	\lambda_{k}(\xi)\geq \frac{(n-1)\lambda_{n}(\xi)+
		(k-1)\widehat{\lambda}_{n}(\xi)+n-k}{(n-1)\lambda_{n}(\xi)-(k-1)\widehat{\lambda}_{n}(\xi)+n+k-2}.
	\end{equation}
\end{theorem}

\begin{remark} \label{reh} 
	For every $n\geq 1$ and transcendental real $\xi$ we have $\widehat{\lambda}_{n}(\xi)\leq \frac{2}{n+1}$, in fact
    \begin{equation} \label{eq:rlsc}
    \widehat{\lambda}_{3}(\xi)< 0.4246, \qquad 
    \widehat{\lambda}_{n}(\xi)\leq \frac{2}{n+1} \quad \text{if}\; n\; \text{odd},
    \qquad \widehat{\lambda}_{n}(\xi)<\frac{2}{n+1} \quad 
    \text{if}\; n\; \text{even},
    \end{equation}
    follows from Roy~\cite{roy2}, Laurent~\cite{lau} and Schleischitz~\cite[Section~4]{ichann}
    respectively. See also~\cite{davsh},~\cite{ichrelations}.
	If $\xi$ satisfies $\lambda_{n}(\xi)>1$ then $\widehat{\lambda}_{n}(\xi)=1/n$,
	see~\cite{schlei}.
\end{remark}

\begin{remark} \label{hirsch}
	We could derive from \eqref{eq:wanndenn} with 
	$\widehat{\lambda}_{n}(\xi)\geq 1/n$ that 
	$\lambda_{k}(\xi)\geq
	(n\lambda_{n}(\xi) + n - k + 1)/(n\lambda_{n}(\xi) + k + n - 1)$.
	However, this turns out not to be of interest as it never both
	exceeds the bound in Theorem~\ref{juppy} and $1/k$.
\end{remark}

For $\widehat{\lambda}_{n}(\xi)=1/n$ the claim
\eqref{eq:folgerunge0} becomes just 
Theorem~\ref{juppy}, otherwise we get a stronger result.
Thereby in particular we provide a new
proof of Theorem~\ref{juppy} that is significantly different 
from the one given in~\cite{badbug}, and from the proofs of
\eqref{eq:herz}, \eqref{eq:jippy}. 
Our proof 
is based on Minkowski's Theorem that we apply in the formalism of
parametric geometry of numbers, and
Mahler's Theorem
on polar convex bodies. In particular, a novelty in our approach are its
close ties to the dual linear form problem intorduced in \eqref{eq:lammda}. 
%We can infer a small metric refinement of~\cite[Theorem~2.1]{badbug}.
%
%\begin{corollary}
%	Let $n\geq 2$ be an integer. Then for $\lambda \geq (n+4)/(3n)$ the
%	set of real $\xi$ that satisfy both properties
	%
%	\[
%	\lambda_{n}(\xi)= \lambda, \qquad %\widehat{\lambda}_{n}(\xi)=\frac{1}{n}
%	\]
	%
%	has Hausdorff dimensions $2n^{-1}(1+\lambda)^{-1}$.
%\end{corollary} 
%
%The proof follows from \eqref{eq:folgerunge0} with a standard measure
%theoretic argument. 
%Without condition $\widehat{\lambda}_{n}(\xi)=\frac{1}{n}$ the statement
%was obtained in~\cite{badbug}.

We enclose a few more remarks on \eqref{eq:folgerunge0}, \eqref{eq:wanndenn}.
In view of the last claim in Remark~\ref{reh}, in \eqref{eq:folgerunge0}
we need $\lambda_{n}(\xi)\leq 1$ for an improvement of Theorem~\ref{juppy}.
As indicated above, by \eqref{eq:folgerunge0}
equality in \eqref{eq:jippy} implies the identities
\[
\widehat{\lambda}_{i}(\xi)=\frac{1}{i}, \qquad\qquad n\leq i\leq k,
\]
a new result in this generality.
If $\lambda_{k}(\xi)>1$, it is already implied
by~\cite[Theorem~1.12]{schlei}, in fact its proof shows the analogous claim
up to $i=2k-1$. 
%It is further not hard to see and follows from~\cite{badbug}
%that equality in \eqref{eq:jippy} further implies the identities
%$\lambda_{i}(\xi)= (n\lambda_{n}(\xi)+n-i)/i$ for 
%$n\leq i\leq k$.
The special case $n=2$ is of particular interest,
as only then numbers $\xi$ satisfying $\widehat{\lambda}_{n}(\xi)>1/n$
have been found. 
Classical examples are extremal
numbers as defined by Roy~\cite{roy1} and 
Sturmian continued fractions, see Bugeaud, Laurent~\cite{buglau},
we omit definitions here. See also Poels~\cite{poels}.
For $n=2, k=3$ and $\xi$ an extremal number, 
the identities from~\cite{roy1},~\cite{ichlondon}
\begin{equation}  \label{eq:ensf}
\lambda_{2}(\xi)=1, \qquad \widehat{\lambda}_{2}(\xi)=\frac{\sqrt{5}-1}{2}=0.6180\ldots, \qquad \lambda_{3}(\xi)=\frac{1}{\sqrt{5}}=0.4472\ldots,
\end{equation}
induce equality in both \eqref{eq:folgerunge0}
and \eqref{eq:wanndenn}. While some extremal numbers are Sturmian
continued fractions,
the identity does not extend to other Sturmian continued fractions $\xi$,
nor does it to larger values of $k$.
For $n=2, k=4$ and $\xi$ an extremal number, \eqref{eq:wanndenn}
still provides a non-trivial bound that reads
\begin{equation} \label{eq:jox}
\lambda_{4}(\xi)\geq \frac{6\sqrt{5}-5}{31}= 0.2715\ldots> \frac{1}{4}.
\end{equation}
However, a stronger bound $\lambda_{4}(\xi)\geq (\sqrt{5}-1)/4= 0.3090\ldots$ with conjectured identity was established 
in~\cite{ichlondon}.
Moreover, we may alternatively derive \eqref{eq:jox} 
from \eqref{eq:bbs2} below.

Combining \eqref{eq:wanndenn} with an inequality of Jarn\'ik~\cite{ja1}
that reads
\begin{equation} \label{eq:mam}
\lambda_{2}(\xi)\geq \frac{\widehat{\lambda}_{2}(\xi)^2}{1-\widehat{\lambda}_{2}(\xi)}
\end{equation}
we can formulate a bound for $\lambda_{k}$ 
only in terms of 
$\widehat{\lambda}_{2}$ that is not trivial when $k\in \{3,4\}$.

\begin{corollary}  \label{einkor}
	For any $\xi$ we have
	%\[
	%\lambda_{k}(\xi)\geq %\frac{(2-k)\widehat{\lambda}_{2}(\xi)^2+(2k-3)\widehat{\lambda}_{2}(\zet%a)+2-k}{k\widehat{\lambda}_{2}(\xi)^{2}+(1-2k)\widehat{\lambda}_{2}(\zet%a)+k}, \qquad k\geq 2.
	%\]
	%In particular
	%
	\begin{equation}  \label{eq:lov}
	\lambda_{3}(\xi)\geq \frac{-\widehat{\lambda}_{2}(\xi)^2+3\widehat{\lambda}_{2}(\xi)-1}{3\widehat{\lambda}_{2}(\xi)^2-5\widehat{\lambda}_{2}(\xi)+3}
	\end{equation}
	and
	\begin{equation}  \label{eq:lov2}
	\lambda_{4}(\xi)\geq \frac{-2\widehat{\lambda}_{2}(\xi)^2+5\widehat{\lambda}_{2}(\xi)-2}{4\widehat{\lambda}_{2}(\xi)^2-7\widehat{\lambda}_{2}(\xi)+4}.
	\end{equation}
\end{corollary}

If $\xi$ is an extremal number, \eqref{eq:lov} becomes an identity
again according to \eqref{eq:ensf}, and  \eqref{eq:lov2} becomes \eqref{eq:jox}. The bound \eqref{eq:lov} exceeds $1/3$
if $\widehat{\lambda}_{2}(\xi)>(7-\sqrt{13})/6=0.5657\ldots$,
and \eqref{eq:lov2} exceeds $1/4$ for $\widehat{\lambda}_{2}(\xi)>(9-\sqrt{17})/8=0.6096\ldots$,
just slightly below the maximum possible value in \eqref{eq:ensf}
obtained for extremal numbers~\cite{roy1}.
Corollary~\ref{einkor} is stronger than what can be derived from 
combining \eqref{eq:folgerunge0} with \eqref{eq:mam}.
Similar bounds for $\lambda_{k}(\xi)$ in terms
of $\widehat{\lambda}_{n}(\xi)$ can be obtained for $n>2$ as well via the 
implicit 
estimates \eqref{eq:nonedtor} below that originate in~\cite{mamo} and generalize \eqref{eq:mam},
but their formulation becomes cumbersome. 
 
 %\begin{theorem}  \label{simto}
 %	For any $\xi$ we have
 	%
 %	\begin{equation}  \label{eq:siii}
 %	\lambda_{k}(\xi)\geq %\frac{(3-k)\widehat{\lambda}_{2}(\xi)^{2}+2(k-2)\widehat{\lambda}_{2}(\xi%)+2-k}{(1-\widehat{\lambda}_{2}(\xi))\cdot %(k-1+(2-k)\widehat{\lambda}_{2}(\xi))}, \qquad k\geq 2.
 %	\end{equation}
 	%
 %	In particular 
 	%
 %	\begin{equation}  \label{eq:ooi}
 %	\lambda_{3}(\xi)\geq %\frac{2\widehat{\lambda}_{2}(\xi)-1}{(1-\widehat{\lambda}_{2}(\xi))\cdot %(2-\widehat{\lambda}_{2}(\xi))}.
% 	\end{equation}
 	%
 %\end{theorem}
%
%Again there is equality for $k=3$ and extremal numbers $\xi$.
%In view of any $\xi$ satisfying
%$\widehat{\lambda}_{2}(\xi)\leq (\sqrt{5}-1)/2$,
%Theorem~\ref{simto} becomes trivial if $k\geq 4$, so only 
%\eqref{eq:ooi} is of interest.

\subsection{Comparison~\eqref{eq:folgerunge0} vs \eqref{eq:wanndenn}} 

We discuss when \eqref{eq:wanndenn} both improves on \eqref{eq:folgerunge0} and exceeds the trivial bound $1/k$.
In  view of Remark~\ref{reh} and Remark~\ref{hirsch}, we may assume
$\widehat{\lambda}_{n}(\xi)>1/n$ and $\lambda_{n}(\xi)\leq 1$.
We take into account the estimates \eqref{eq:rlsc} and
distinguish 3 cases.

\begin{itemize}
	\item \underline{Case 1}: $k=2n-1$. If  $\lambda_{n}(\xi)=1$,
the bounds in \eqref{eq:wanndenn} and \eqref{eq:folgerunge0} coincide, regardless of the value of $\widehat{\lambda}_{n}(\xi)$. If  $\lambda_{n}(\xi)<1$, then the bound \eqref{eq:wanndenn} is stronger
for any $\widehat{\lambda}_{n}(\xi)$. Moreover, \eqref{eq:wanndenn}
is non-trivial if
\begin{equation}  \label{eq:getss}
\widehat{\lambda}_{n}(\xi)> \frac{n+1+(1-n)\lambda_{n}(\xi)}{2n}.
\end{equation}
Thus for instance if $\lambda_{n}(\xi)=1$ and $\widehat{\lambda}_{n}(\xi)>1/n$, both are guaranteed.
As we decrease $\lambda_{n}(\xi)$,
condition \eqref{eq:getss} on $\widehat{\lambda}_{n}(\xi)$ becomes
more stringent, if $\lambda_{n}(\xi)\leq 1-\frac{2}{n-1}$
then $\widehat{\lambda}_{n}(\xi)\geq 2/n$ which contradicts \eqref{eq:rlsc}. 
	So for \eqref{eq:wanndenn}
to be interesting, we require
\[
\lambda_{n}(\xi)\in (1-\frac{2}{n-1},1],
\]
in fact a slightly larger lower bound can be stated.
%If either $k>2n-1$ and $\lambda_{n}(\xi)=1$, or $k=2n-1$ and %$\lambda_{n}(\xi)<1$,
%\eqref{eq:wanndenn} improves on Theorem~\ref{gutersatze0}.
 
 \item
\underline{Case 2}: $k<2n-1$. The claim \eqref{eq:wanndenn} turns out stronger 
than \eqref{eq:folgerunge0} as soon as
\begin{equation} \label{eq:A1}
\widehat{\lambda}_{n}(\xi) > \frac{(1-n)\lambda_{n}(\xi)+k-n}{k-2n+1},
\end{equation}
(or $\lambda_{n}(\xi)< \frac{k-n}{n-1}$ but this is of no interest here)
and non-trivial if
\begin{equation}  \label{eq:A2}
\widehat{\lambda}_{n}(\xi) > 
\frac{(1-n)\lambda_{n}(\xi)+k-n+2}{k + 1}.
%\frac{k(1-n)\lambda_{n}(\xi) + k^2 -kn+k- 1}{(k - n)(k + 1)}.
\end{equation}
The bound in \eqref{eq:A1} rises with respect to $\lambda_{n}(\xi)$ whereas
\eqref{eq:A2} decays, and they coincide for $\lambda_{n}(\xi)=(k-n+1)/n<1$
which yields $1/n$ in both expressions.
Thus as soon as $\lambda_{n}(\xi)\in((k-n+1)/n,1]$ we only have
to satisfy \eqref{eq:A1}, i.e. $\widehat{\lambda}_{n}(\xi)>c>1/n$ 
where $c$ depends on $\lambda_{n}(\xi)$,
and have both requirements met. It can be checked that
here we do not get
any restrictions from \eqref{eq:rlsc} under our assumptions above.
% EXPLANATION!!!!    as
%only for
%
%\begin{equation} \label{eq:contar}
%\lambda_{n}(\xi) \leq \frac{  kn-2k + 2n - n^2 - 2}{n^2 - n}
%\end{equation}
%
%condition \eqref{eq:A2} would give a contradiction, however it can be checked 
%that \eqref{eq:contar} contradicts
%our assumption $\lambda_{n}(\xi)>(k-n+1)/n$.

\item
\underline{Case 3}: $k>2n-1$. Then for improving \eqref{eq:folgerunge0},
conversely to \eqref{eq:A1} we need 
\begin{equation} \label{eq:B1}
\widehat{\lambda}_{n}(\xi) < \frac{(1-n)\lambda_{n}(\xi)+k-n}{k-2n+1},
\end{equation}
whereas the condition \eqref{eq:A2} for non-triviality remains unchanged.
A discussion of when \eqref{eq:A2}, \eqref{eq:B1} 
can occur simultaneously, upon taking into account $\lambda_{n}(\xi)\leq 1$ and \eqref{eq:rlsc}, after some calculation
finally implies the necessary conditions
\begin{equation}  \label{eq:satt}
k=2n, \qquad  \; n\neq 3,    \qquad
1-\frac{n+2}{n^2-n} < \lambda_{n}(\xi)\leq 1.
\end{equation}
We notice that for
$n=2$, $k=4$ and extremal numbers $\xi$, condition \eqref{eq:satt} is satisfied and an improvement is indeed obtained, 
see \eqref{eq:jox}.
%  EXPLANATION !!!!!! 
%The right hand side of \eqref{eq:B1} now decreases in $\lambda_{n}(\xi)$, %and faster
%than the expression on \eqref{eq:A2}. 
%Notice $\lambda_{n}(\xi)\leq 1<(k-n+1)/n$, the latter 
%induces equilibrium of \eqref{eq:A1}, \eqref{eq:A2} with value $1/n$. 
%So together with condition \eqref{eq:A2},
%for
%$\widehat{\lambda}_{n}(\xi)\in (c,d)$ with certain $d>c>1/n$
%depending on (decaying in) $\lambda_{n}(\xi)$, both is satisfied. % check
%Here for large $k$ the lower bound $c$ exceeds the bounds
%from \eqref{eq:rlsc}. Indeed, by $\lambda_{n}(\xi)\leq 1$ 
%a short calculation shows \eqref{eq:A2} and \eqref{eq:rlsc} 
%imply $k=2n$
%and moreover, for $n=3$, in fact 
%Roy's bound in \eqref{eq:rlsc} shows that Case 3 never
%gives both claims. So in Case 3 for \eqref{eq:wanndenn}
%to exceed \eqref{eq:folgerunge0} we require
%
%\[
%k=2n, \qquad  \; n\neq 3,    \qquad
%1-\frac{n+2}{n^2-n} < \lambda_{n}(\xi)\leq 1,
%\]
%
%where the last lower bound comes from combining \eqref{eq:uschranke}, %\eqref{eq:A2}
%and can be sharpened using \eqref{eq:rlsc} instead.
\end{itemize}

%\begin{remark}
%Jarnik identity and another
%well-known estimate of Jarnik (see also Schmidt-Summerer Monatsh.)
%give
%
%\[
%w_{3}\geq w_{2}\geq \widehat{w}_{2}(\widehat{w}_{2}-1)=
%\frac{\widehat{\lambda}_{2}}{(1-\widehat{\lambda}_{2})^{2}}>3
%\]
%
%or equivalently $\lambda_{3}>1/3$
%as soon as $\widehat{\lambda}_{2}>(7-\sqrt{13})/6= 0.5657...$,
%however independently from $\lambda_{2}$ (without requirement
%$\lambda_{2}(\xi)=1$).
%\end{remark}

\section{Relations involving simultaneous and linear form exponents} \label{mix}

\subsection{Mixed properties}
Now we want to find relations that also contain linear form exponents
$w_{N}(\xi), \widehat{w}_{N}(\xi)$. The following relation
was already implictly derived in~\cite{ichann} 
and obtained with a different proof and explicitly formulated
by Badziahin, Bugeaud~\cite{badbug}.

\begin{theorem}[Badziahin, Bugeaud; Schleischitz] \label{bbyd}
	Let $k\geq n\geq 1$ be integers and 
	$\xi$ be a real number. We have
	\begin{equation} \label{eq:bbs2}
	\lambda_{k}(\xi)\geq \frac{ w_{n}(\xi)-k+n}{(n-1)w_{n}(\xi)+k}.
	\end{equation}
\end{theorem}

The bound exceeds $1/k$ iff $w_{n}(\xi)>k$. For consequences of Theorem~\ref{bbyd} 
regarding the Hausdorff dimensions of the level sets 
$
\{\xi\in\mathbb{R}: \lambda_{N}(\xi)\geq \lambda\}$ and
$\{\xi\in\mathbb{R}: \lambda_{N}(\xi)= \lambda\}$
for $\lambda\in[1/N,\infty]$,
see~\cite{badbug}. In this note we put no emphasis
on metrical aspects, see however Section~\ref{metrisch}.
We remark that combining \eqref{eq:bbs2} for $n=2$ and $k\in\{3,4\}$ with
Jarn\'ik's identity~\cite{jahrnik} and another estimate of Jarn\'ik~\cite{ja2} given as
\begin{equation} \label{eq:jhr}
w_{2}(\xi)\geq \widehat{w}_{2}(\xi)^2-\widehat{w}_{2}(\xi), \qquad
\widehat{w}_{2}(\xi)=\frac{1}{1-\widehat{\lambda}_{2}(\xi)}
\end{equation}
yields another proof of Corollary~\ref{einkor}. In particular, for $n=2, k=3$,
Theorem~\ref{bbyd} is again sharp when $\xi$ is an extremal number,
and also for any Sturmian continued fraction $\xi$ as
follows from~\cite{ichjnt}.
We complement Theorem~\ref{bbyd} with inequalities
containing uniform exponents again. Our first estimate reads as follows.

\begin{theorem} \label{t2}
	Let $k\geq n\geq 2$ be integers and 
	$\xi$ be a real number. We have
	\begin{equation}  \label{eq:steuern}
	\lambda_{k}(\xi) \geq  
	\frac{w_{n}(\xi)\widehat{w}_{n}(\xi)-
		w_{n}(\xi)+(n-k)\widehat{w}_{n}(\xi)}{(n-2)w_{n}(\xi)\widehat{w}_{n}(\xi)+w_{n}(\xi)
		+(k-1)\widehat{w}_{n}(\xi)}.
	\end{equation}
	%
	%with
	%
	%\[
	%A=(l-1)(l+1)+(k-l+1)(k+1)+(k+1)-(k-l+2)(k-l)=(k+1)(l+1)
	%\]
	%
	%and
	%
	%\[
	%B=(l-1)(l+1)-(k-l+1)(k+1)l-(k+1)l+(k+1)(l+1)-(k-l+2)(k-l)
	%=(l-k)(k+1)(l+1)
	%\]
	%
	%and
	%
	%\[
	%C=k(l-1)(l+1)-(k-l+1)(k+1)-(k+1)+(k-l+2)(k-l)=(l^2-l-2)(k+1)
	%\]
	%
	%and
	%
	%\[
	%D=k(l-1)(l+1)+(k-l+1)(k+1)l+(k+1)l-(k+1)(l+1)+(k-l+2)(k-l)
	%=(k^2-1)(l+1).
	%\]
	%
\end{theorem}

%In the special case $w_{n}=\widehat{w}_{n}$ we would 
%obtain the transference inquality
%
%\[
%\lambda_{
%\]
%

The special case $k=n$ simplifies to an estimate of
Bugeaud and Laurent~\cite{bl2010}, i.e.
\begin{equation}  \label{eq:this}
\lambda_{k}(\xi)\geq \frac{(\widehat{w}_{k}(\xi)-1)w_{k}(\xi)}{((k-2)\widehat{w}_{k}(\xi)+1)w_{k}+(k-1)\widehat{w}_{k}(\xi)}, \qquad k\geq 2.
\end{equation}
See also~\cite{blau07},~\cite{laurent}, and also Schmidt and Summerer~\cite{ssch} for another proof of \eqref{eq:this}.
As in~\cite{bl2010},~\cite{laurent},~\cite{ssch}, then our proof applies to the more general setting of
$\mathbb{Q}$-linearly independent
$\{ 1,\xi_{1},\ldots,\xi_{k}\}$, so with respect to the exponents $\lambda_{k}(\underline{\xi}), w_{k}(\underline{\xi}),
\widehat{w}_{k}(\underline{\xi})$ from Section~\ref{introd}.

We compare \eqref{eq:steuern} with \eqref{eq:bbs2}. In contrast
to \eqref{eq:folgerunge0}, here we only improve on Badziahin, Bugeaud
in certain cases. As we explain below, it turns out this may 
happen in the cases
\begin{equation}  \label{eq:Z1}
\text{\underline{Case 1}:}\;\; n\geq 3, \quad w_{n}(\xi)=w_{n-1}(\xi), 
\qquad \text{\underline{Case 2}:}\;\; n\geq 3, \quad n+1\leq k\leq 2n-2.
\end{equation}

A short calculation shows that our new result is stronger 
than Theorem~\ref{bbyd} as soon as
\begin{equation} \label{eq:ason}
\widehat{w}_{n}(\xi) > \frac{nw_{n}(\xi)}{w_{n}(\xi)-k+n}.
\end{equation}
%
%or equivalently if the denominor is positive (which is true if $k<2l$) and
%\[
%k< w_{l}(\xi)+l-l\frac{w_{l}(\xi)}{\widehat{w}_{l}(\xi)}.
%\]
%In particular, if $k=l$ (Laurent's result) then we only require %$\widehat{w}_{l}(\xi) > l$.
We elaborate on how restrictive this estimate is.
First we notice that \eqref{eq:ason} enables the trivial condition $\widehat{w}_{n}(\xi)\leq w_{n}(\xi)$
as soon as $w_{n}(\xi)>k$ (a slightly more restrictive bound was obtained in~\cite{mamo}), which we impose anyway for a non-trivial estimate. 
Assume $\xi$ satisfies
\begin{equation} \label{eq:assu}
w_{n}(\xi)>w_{n-1}(\xi).
\end{equation}
Then another restriction comes from the reverse estimate of the form
\begin{equation}  \label{eq:bush}
\widehat{w}_{n}(\xi) \leq \frac{nw_{n}(\xi)}{w_{n}(\xi)-n+1}
\end{equation}
of~\cite[Theorem~2.2]{buschlei}.
Hence, according to \eqref{eq:ason}
in this case Theorem~\ref{t2} may improve on \eqref{eq:bbs2}
only if $k<2n-1$, while it at best confirms the same bound if $k=2n-1$.
It can be shown that for $n=2$ and $\widehat{w}_{2}(\xi)>2$
we have \eqref{eq:assu} automatically satisfied, see Proposition~\ref{propper} below. Combination leaves the cases \eqref{eq:Z1} open
for potential improvement. Since the existence of $\xi$ with $\widehat{w}_{n}(\xi)>n$ 
for any $n>2$ is at present unproved, we cannot yet provide numbers for which Theorem~\ref{t2} improves Theorem~\ref{bbyd}.
 
Let now $n=2$ and $\xi$ be a Sturmian continued fraction.
This setup induces equality in \eqref{eq:bush} as can be seen 
from the main result of~\cite{buglau}. Then for $k=3$
we once more obtain the correct value of $\lambda_{3}(\xi)$
from~\cite{ichjnt} 
as a lower bound, so
Theorem~\ref{t2} is sharp in certain cases as well. For $k>3$
the bound \eqref{eq:bbs2}
is stronger than \eqref{eq:steuern}. 
%since for $n=2$ and
%Sturmian continued fractions \eqref{eq:assu} holds. 

In the case $w_{n}(\xi)=\infty$,
Theorem~\ref{t2} yields
$\lambda_{k}(\xi)\geq (n-1)^{-1}$, confirming
a partial claim of~\cite[Theorem~2.1]{ichann} stating
that if $w_{n-1}(\xi)<\infty$, i.e. $\xi$
is a $U_{n}$-number in Mahler's classification, then
$\lambda_{k}(\xi)= (n-1)^{-1}$
for large
enough $k$. On the other hand, in case of $\widehat{w}_{n}(\xi)>n$ the bound
\eqref{eq:steuern} will exceed $1/(n-1)$ for every $k$.
This leads to a new proof that $U_{n}$-numbers satisfy $\widehat{w}_{n}(\xi)=n$,
already obtained in~\cite[Corollary~2.5]{buschlei}.
%From this dichotomy we get that if $w_{l}(\xi)=\infty$ then %$\widehat{w}_{l}(\xi)=l$. (Upon the condition $w_{l-1}(\xi)<\infty$
%this is already a consequence of [cite Bugeaud, Schleischitz Cor~2.5].)
Adamczewski, Bugeaud~\cite{adambug} showed the related claim that $\widehat{w}_{n}(\xi)> n$
for some $n\geq 1$ implies $\xi$ is no $U_{k}$-number with $k>n$
(nor $k=n$ as pointed out above), see also Roy~\cite{roy5} when $n=2$.
%\begin{corollary}
%    Any $U$-number $\xi$ satisfies
%	\[
%	\widehat{w}_{k}(\xi)= k, \qquad k\geq 1.
%	\]
%	In other words, if a real number 
%	$\xi$ satisfies $\widehat{w}_{l}(\xi)> l$
%	for some $l\geq 1$ then $\xi$ is an $S$-number or a $T$-number.
%\end{corollary}
%\begin{proof}
%	Assume $\xi$ is a $U_{l}$-number for some index $l$. 
%	Then by our above findings
%	we have 
%	\begin{equation} \label{eq:ab}
%	\widehat{w}_{k}(\xi)= k, \qquad k\geq l.
%	\end{equation}
%	On the other hand, it was shown by \thetaczewski and Bugeaud [cite]
%	that if $\xi$ is real with $\widehat{w}_{k}(\xi)>k$ then 
%	$\xi$ is not a $U$-number of index greater than $l$. Rephrasing
%	this means we get \eqref{eq:ab} for $\xi$ and indices $1\leq k<l$ as %well.  
%	\end{proof}
We finally remark
that $\widehat{w}_{k}(\xi)\leq k+n-1$ holds for any $U_{n}$-number
$\xi$ and every $k\geq 1$ by~\cite[Corollary~2.5]{buschlei}, 
for $n\leq k\leq 2n-2$ see
also the bound from~\cite[Corollary~2.3]{ichfunc}.

In some cases,
we can strengthen our estimates in Theorem~\ref{t2}
upon the additional assumption
\eqref{eq:assu} on $\xi$.

\begin{theorem}  \label{th3}
	Let $k,n$ be integers with $2\leq n\leq k\leq 2n-2$.
	Assume $\xi$ is a real
	number and satisfies the inequality
	\eqref{eq:assu}.
	Then
	\begin{equation}  \label{eq:tarda}
	\lambda_{k}(\xi)\geq \frac{w_{n}(\xi)\widehat{w}_{n}(\xi)+
		(n-k-1)w_{n}(\xi)+(n-k)\widehat{w}_{n}(\xi)}{(2n-k-2)w_{n}(\xi)\widehat{w}_{n}(\xi)+(k-n+1)w_{n}(\xi)
		+(n-1)\widehat{w}_{n}(\xi)}.
	\end{equation}
\end{theorem}

\begin{remark}
	A variant for $k>2n-2$ turns out weaker
	than Theorem~\ref{bbyd}.
\end{remark}

If $k=n$ we again obtain formula \eqref{eq:this} as from Theorem~\ref{t2}, so then condition \eqref{eq:assu} is not required. Otherwise 
\eqref{eq:tarda} is stronger than \eqref{eq:steuern}.
Assumption \eqref{eq:assu} may not be required for the conclusion, 
however in our proof as in~\cite{buschlei} 
it guarantees some nice properties of the
integer polynomials realizing the exponent $w_{n}(\xi)$, 
see Proposition~\ref{propper} below.
Theorem~\ref{th3} improves on Theorem~\ref{bbyd} upon
the same condition \eqref{eq:ason}, thus we require Case 2 of \eqref{eq:Z1}
and still cannot settle existence of numbers $\xi$ where this happens.

Our last claim provides a lower bound for $\lambda_{k}(\xi)$ in terms
of $\widehat{w}_{n}(\xi)$ only, if $n\leq k\leq 2n-2$. We also
include a bound that arises as a hybrid with Theorem~\ref{th3}.

\begin{theorem}  \label{gleich}
	Let $k\geq n\geq 2$ be integers  
	and $\xi$ be a transcendental real number. Then
	\begin{equation}  \label{eq:jjj}
	\lambda_{k}(\xi) \geq \frac{ \widehat{w}_{n}(\xi)+2n-2k-1 }{ (2n-k-2)\widehat{w}_{n}(\xi)+k }, \qquad\quad \text{if}\;\; k\leq 2n-2.
	\end{equation}
	In fact we have the stronger bound
	\begin{equation}  \label{eq:AA1}
	\lambda_{k}(\xi) \geq \min \left\{ \Theta \; , \; \frac{\widehat{w}_{n}(\xi)+2n-2k- 2}{(2n-k-3) \widehat{w}_{n}(\xi)+k}  \right\}, \qquad\quad \text{if}\;\; k\leq 2n-3,
	\end{equation}
	where $\Theta$ denotes the bound in \eqref{eq:tarda}.
\end{theorem}

\begin{remark}  \label{reh4}
	The proof method of Theorem~\ref{gleich}
	also provides a new proof of
	\begin{equation} \label{eq:njor}
	\widehat{w}_{n}(\xi)\leq 2n-1, \qquad\qquad n\geq 1,
	\end{equation}
	due to Davenport and Schmidt~\cite{davsh}, 
	corresponding to the case $k=2n-1$.
	See~\cite{buschlei},~\cite{ichacta} for slightly stronger bounds,
	and~\cite{davsh},~\cite{roy1} for $n=2$.
	Unfortunately, combining \eqref{eq:jjj} with German's estimates \eqref{eq:ogerman} below
	turns out not to give an interesting relation between $\lambda_{k}$ and
	$\widehat{\lambda}_{n}$ in view of \eqref{eq:rlsc}.
\end{remark}

The bound \eqref{eq:jjj} is non-trivial, i.e. gives $\lambda_{k}(\xi)>1/k$,
as soon as $\widehat{w}_{n}(\xi)>k$. So we restrict to this case
in the sequel (which requires $n\geq 3$ if $k>n$).
The bound \eqref{eq:jjj} is smaller than both expressions in \eqref{eq:AA1} if $w_{n}(\xi)>\widehat{w}_{n}(\xi)$,
in particular weaker than the conditional bound \eqref{eq:tarda}.
We compare \eqref{eq:jjj} with the unconditional Theorem~\ref{bbyd} and Theorem~\ref{t2}.
A short computation shows that
it improves Theorem~\ref{bbyd} if
\begin{equation}  \label{eq:rhside}
w_{n}(\xi) < \frac{(2n-k)\widehat{w}_{n}(\xi) - k}{ 2n - 1 - \widehat{w}_{n}(\xi)}.
\end{equation}
For $\widehat{w}_{n}(\xi)=k$ the right hand side
gives $k$ as well, and
it exceeds $\widehat{w}_{n}(\xi)$ if  $\widehat{w}_{n}(\xi)>k$.
Marnat, Moshchevitin~\cite{mamo} generalized \eqref{eq:jhr}
by improving the trivial estimate $w_{n}(\xi)\geq \widehat{w}_{n}(\xi)$
for $n\geq 2$, which
plays against \eqref{eq:rhside}.
Nevertheless, \eqref{eq:jjj} is potentially of interest in many cases.
Assume $1<\alpha<\beta<2$ are fixed,
and for large $n$ choose
$k=\alpha n+o(n)$ and $\widehat{w}_{n}(\xi)=\beta n+o(n)$. 
It can be checked that then~\cite{mamo} only gives a lower bound $w_{n}(\xi)/\widehat{w}_{n}(\xi)\geq 1+o(1)$ as $n\to\infty$.
When neglecting lower order terms, we see that
for $w_{n}(\xi)/\widehat{w}_{n}(\xi)<(2-\alpha)/(2-\beta)+o(1)$ as $n\to\infty$, property \eqref{eq:rhside}
will be satisfied. This leaves a non-empty interval $(1+o(1),(2-\alpha)/(2-\beta)+o(1))$ for the ratio, for large $n$.
See also Example~\ref{ex} below.

Inequality~\eqref{eq:jjj} improves Theorem~\ref{t2}
as soon as 
\begin{equation}  \label{eq:bdbd}
\frac{w_{n}(\xi)}{\widehat{w}_{n}(\xi)} < \frac{ (3kn-k^2-k-2n^2+2n-1)\widehat{w}_{n}(\xi) + kn-k^2+k-2n+1}{ (n-k)\widehat{w}_{n}(\xi)^{2} + (3n-2k-1-2n^2+2kn)\widehat{w}_{n}(\xi) + k+1-2n   }.
\end{equation}
Despite~\cite{mamo} recalled above,
the scenario that \eqref{eq:bdbd} holds for many $\xi$ and
$k,n$ is likely. With $\alpha, \beta$ as above,
for  $w_{n}(\xi)/\widehat{w}_{n}(\xi)<(3\alpha-\alpha^{2}-2)/(\beta-\alpha\beta+2\alpha-2)+o(1)$ as $n\to\infty$, we satisfy \eqref{eq:bdbd}.
We want to mention that \eqref{eq:bdbd} requires 
$k>n$, indeed for $k=n$ the bound 
in \eqref{eq:jjj} becomes \eqref{eq:this} if 
$w_{n}(\xi)=\widehat{w}_{n}(\xi)$ and is weaker otherwise.
The next example illustrates a hypothetical 
scenario where Theorem~\ref{gleich} is 
reasonably strong.

\begin{example}  \label{ex}
	Let $n=10, k=13$ and assume $\xi$ is a real number satisfying $\widehat{w}_{10}(\xi)=14$. Then~\cite{mamo} gives $w_{10}(\xi) \geq 15.0190\ldots$. 
	The right hand sides of \eqref{eq:rhside} and \eqref{eq:bdbd} 
	become $17$ and
	$16.1875\ldots$, respectively. Hence $\lambda_{13}(\xi)\geq 7/83=0.0843\ldots$ from Theorem~\ref{gleich} improves
	on both Theorem~\ref{bbyd} and Theorem~\ref{t2} if
	$w_{10}(\xi)\in (15.0190,16.1875)$.
\end{example}

For $k=2n-2$, the bound \eqref{eq:jjj} becomes an easy affine function
	\[
	\lambda_{2n-2}(\xi) \geq \frac{ \widehat{w}_{n}(\xi)-2n+3 }{ 2n-2}.
	\]
	This may be of interest for $3\leq n\leq 9$. If $n\geq 10$ then $\widehat{w}_{n}(\xi)\leq 2n-2$ for any $\xi$ was established in~\cite{ichacta}
(see also~\cite{buschlei}), and the bound becomes trivial. For $n=2$ the implied bound $\lambda_{2}(\xi) \geq 
(\widehat{w}_{2}(\xi)-1)/2$ is weaker than $\lambda_{2}(\xi) \geq 
\widehat{w}_{2}(\xi)-2+\widehat{w}_{2}(\xi)^{-1}$ derived from Jarn\'ik's identity \eqref{eq:jhr} and \eqref{eq:mam}. We close this section with
an asymptotic result.

\begin{corollary}
	Let $\xi$ be a real transcendental number and write
	\[
	\overline{\widehat{w}}(\xi)= \limsup_{n\to\infty} \frac{ \widehat{w}_{n}(\xi) }{n}, \qquad\qquad \overline{\lambda}(\xi)= \limsup_{n\to\infty} n\lambda_{n}(\xi).	
	\]
	Then
	\[
	\overline{\lambda}(\xi) \geq   \frac{ \left(2-\sqrt{2 - \overline{\widehat{w}}(\xi) }\right) \cdot \left(\overline{\widehat{w}}(\xi) + 2\sqrt{2 - \overline{\widehat{w}}(\xi)}- 2\right)}{\overline{\widehat{w}}(\xi)\sqrt{2 - \overline{\widehat{w}}(\xi)}}.
	\]
	The according estimate with respect to the lower limits holds as well.
\end{corollary}

The bound as a function of $\overline{\widehat{w}}(\xi)$
induces an increasing bijection of the interval $[1,2]$ onto itself, upon taking the left-sided limit if $\overline{\widehat{w}}(\xi)=2$.
It can be seen complementary to $\overline{\lambda}(\xi) \geq (\overline{w}(\xi)+1)^2/(4\overline{w}(\xi))$  from~\cite[Theorem~2.1]{ichann}, for  $\overline{w}(\xi)$ defined analogously  with respect to ordinary exponents. The latter estimate can be derived from Theorem~\ref{bbyd}.

\begin{proof}
	Choose $k=\lfloor(2-\sqrt{2-\widehat{w}_{n}(\xi)/n })\cdot n\rfloor$ in \eqref{eq:jjj}
	and look at the dominant terms as $n\to\infty$, we skip details.
	\end{proof}

\subsection{Metrical considerations} \label{metrisch}
    As a small metrical application of our results,  we discuss the problem of estimating the Hausdorff dimensions of
    \[
    	\{ \xi:  \widehat{w}_{n}(\xi) \geq \widehat{w} \}, \qquad\qquad 
    	\widehat{w}\in [n,2n-1).
    \] 
     For simplicity we deal with a normalized problem 
     and consider $n\to\infty$. A well-known metric result of Bernik~\cite{bernik} immediately yields the 
    trivial bound
    \begin{equation}  \label{eq:currkn2}
    \dim \left\{ \xi:  \frac{\widehat{w}_{n}(\xi)}{n} \geq \beta \right\}
    \leq \dim \left\{ \xi:  \frac{w_{n}(\xi)}{n} \geq \beta \right\}
    \leq \frac{1}{\beta}+o(1), \qquad \beta\in [1,2],\;\; n\to\infty.
    \end{equation}
    The estimates from~\cite{mamo} also do not improve this 
    asymptotic relation.
    While \eqref{eq:currkn2} seems a very crude estimate,  
    nothing better seems currently available.
  
	Upon suitable choice of $k$, the inclusion
	\begin{equation}  \label{eq:currkn}
	\{ \xi:  \widehat{w}_{n}(\xi) \geq \widehat{w} \}
	\subseteq \left\{ \xi:  \lambda_{k}(\xi) \geq \frac{ \widehat{w}+2n-2k-1 }{ (2n-k-2)\widehat{w}+k } \right\}, \qquad \widehat{w}>n,
	\end{equation}
	induced by \eqref{eq:jjj} may have potential to improve \eqref{eq:currkn2}, at least in certain parameter ranges for $\beta$.
    Unfortunately, no reasonable upper bounds for 
	the dimensions of level sets $\{ \xi:  \lambda_{k}(\xi) \geq \lambda\}$ for $\lambda\in [1/k,2/k]$ are yet available
	that we would require for this avenue. However, we want
	to give in to some speculation. Assume Beresnevich's~\cite{bere} lower
	bound 
	\begin{equation} \label{eq:currn}
	\dim \{ \xi\in \mathbb{R}: \lambda_{k}(\xi)\geq \lambda\} 
	\geq	\frac{k+ 1}{\lambda+ 1} -(k-1), \qquad\qquad \lambda\in \left[\frac{1}{k}, \frac{3}{2k-1}\right],
	\end{equation}
	is an identity (as conjectured by him and proved for $k=2$) and the reverse estimate
	extends to $\lambda\in [1/k,c/k]$ for some $c$ close to $2$. 
	Then choosing $k=\lfloor(2-\sqrt{2-\beta})n\rfloor$ in order to
	maximize the expression $k\lambda_{k}(\xi)$, 
	indeed it turns out via \eqref{eq:currkn} we improve \eqref{eq:currkn2} for $\beta\in (\frac{17}{9},2-\epsilon)$ with small $\epsilon$,
	for $n$ large enough. We believe our assumption is reasonable, in particular it agrees with the lower bound
	\begin{equation}  \label{eq:bbschronk}
	\dim \{ \xi\in \mathbb{R}: \lambda_{k}(\xi)\}\geq \max_{1\leq N\leq k} \left\{  \frac{(N+1)(1-(N-1)\lambda)}{(k-N+1)(1+\lambda)} \right\}, 
	\qquad\qquad \lambda\geq \frac{1}{n},
	\end{equation}
	from~\cite[Theorem~2.3]{badbug} (also
	obtained in~\cite{ichann}) is $\epsilon$ is small enough. 
	If \eqref{eq:bbschronk} is a good approximation to the true value, 
	we can even extend the above interval for $\beta$, in case of 
	a hypothetical equality 
	in \eqref{eq:bbschronk} (that however contradicts \eqref{eq:currn}
	for $\lambda\leq 3/(2k-1)\approx (3/2)k^{-1}$)
	a calculation verifies
	we get a stronger bound for every $\beta\in [1,2]$. 
	Roughly speaking, Theorem~\ref{gleich} shows that
	not both \eqref{eq:currkn2} and \eqref{eq:bbschronk}
	can be sharp.

\section{The $\mathbb{Q}$-linearly independent case} \label{qli}

For sake of completeness,
we want to formulate similar going-up principles 
for the case of $\mathbb{Q}$-linearly independent real 
vectors. In this situation
we consider extensions of a given
real vector, or equivalently projections of infinite vectors
$\underline{\xi}\in\mathbb{R}^{\mathbb{N}}$
to its first $N$ coordinates, and compare the exponents of approximation
as $N$ increases (note that this is a very different setup
than the going-up principles for fixed $N$ that relate the so-called
intermediate exponents, as for instance in~\cite{blau07}). 
If $\underline{\xi}=(\xi,\xi^2,\xi^3,\ldots)$
we are in the situation of Sections~\ref{se2} and~\ref{mix}. In the
general setting, all results will be considerably weaker,
as may be expected, and the proofs are considerably shorter and easier 
when directly applying well-known transference inequalities. 
The hidden work
in proving these preliminaries appears to some extent in our proofs for
results of Sections~\ref{se2},~\ref{mix}, we elaborate a little 
more on this issue in Section~\ref{pqli}.   
Our first result resembles \eqref{eq:folgerunge0}.

\begin{theorem}  \label{allgemein}
	Let $k\geq n\geq 1$ be integers and
	$\underline{\xi}=(\xi_{1},\xi_{2},\ldots)$ be an
	infinite vector of real numbers.
	For $N\geq 1$,
	denote by $\underline{\xi}_{N}=(\xi_{1},\ldots,\xi_{N})$ the projection of $\underline{\xi}$
	to the first $N$ entries. 
	Assume that $\{1,\xi_{1},\ldots,\xi_{k}\}$ 
	is $\mathbb{Q}$-linearly independent. Then
	\begin{equation}  \label{eq:bo1}
	\lambda_{k}(\underline{\xi}_{k})\geq
	\frac{(n-1)\lambda_{n}(\underline{\xi}_{n})+\widehat{\lambda}_{n}(\underline{\xi}_{n})+n-2 }{ (k-1)(n-1)\lambda_{n}(\underline{\xi}_{n})-\widehat{\lambda}_{n}(\underline{\xi}_{n})+kn-n-k+2 }.
	\end{equation}
	Moreover,
\begin{equation}  \label{eq:bo2}
\lambda_{k}(\underline{\xi}_{k})  \geq \frac{(A-1)B}{((N-2)A+1)B+(N-1)A}
\end{equation}
with
\[
A= \frac{(n-1)(k-1)^2}{nk-k-2n+3-\widehat{\lambda}_{n}(\underline{\xi}_{n})},
\quad B= \frac{(n-1)\lambda_{n}(\underline{\xi}_{n})+\widehat{\lambda}_{n}(\underline{\xi}_{n})+n-2}{1-\widehat{\lambda}_{n}(\underline{\xi}_{n})}.
\]
\end{theorem}

When expanded by inserting for $A,B$, the bound \eqref{eq:bo2} becomes 
a lengthy expression that we omit to 
state explicitly. It 
exceeds \eqref{eq:bo1} as soon as
$\widehat{\lambda}_{n}(\underline{\xi}_{n})>(k-n+1)/n$, which relies
on the fact that we use Theorem~\ref{schlechtsatz} below in the proof.
Since $\widehat{\lambda}_{n}(\underline{\xi}_{n})\geq 1/n$, as a corollary 
of \eqref{eq:bo1} we obtain a variant that resembles Theorem~\ref{juppy}.

\begin{theorem} \label{gutsatz}
	Upon the assumptions of Theorem~\ref{allgemein}, assume 
	\[
	\lambda_{n}(\underline{\xi}_{n})>\frac{k-n+1}{n}.
	\]
	Then 
	\begin{equation}  \label{eq:schwachmat}
	\lambda_{k}(\underline{\xi}_{k})\geq 
	\frac{n\lambda_{n}(\underline{\xi}_{n})+n-1}{n(k-1)(\lambda_{n}(\underline{\xi}_{n})+1)+1}>\frac{1}{k}.
	\end{equation}
\end{theorem} 

For $\lambda_{n}(\underline{\xi}_{n})=(k-n+1)/n$  the right
inequality of \eqref{eq:schwachmat} would become an identity. We believe 
Theorem~\ref{gutsatz} is optimal in the general setting. 
%In contrast to Corollary~\ref{kurve},
We briefly talk about metric consequences, even though
the metric theory with respect to 
the entire space is complete. 
It is known thanks to Jarn\'ik~\cite{jarnik} (see also Dodson~\cite{dodson}) that
for $\lambda\in[1/N,\infty]$ we have
\begin{equation} \label{eq:jarn}
\mathscr{D}_{N}(\lambda):= \dim \{ \underline{\xi}\in\mathbb{R}^{N}: \lambda_{N}(\underline{\xi})\geq \lambda\}=\frac{N+1}{\lambda+1}, \qquad\quad N\geq 1.
\end{equation} 
Theorem~\ref{gutsatz} and the property $\dim(A\times B)\geq \dim(A)+\dim(B)$
of the Hausdorff dimension for $A$ the set in \eqref{eq:jarn} 
with $N=n$ and
 $B=\mathbb{R}^{k-n}$ implies 
\begin{equation} \label{eq:chrom}
\mathscr{D}_{k}\Big(\frac{n\lambda+n-1}{n(k-1)(\lambda+1)+1}\Big) 
\geq \mathscr{D}_{n}(\lambda)+k-n,  \qquad \lambda\geq \frac{k-n(k-n)+1}{n}.
\end{equation}
%
%For any vector $\underline{\xi}_{n}$ in $\mathscr{D}_{n}(\lambda)%$ consider its Cartesian product with $\mathbb{R}^{m}$.
%It suffices to notice that $\dim(A\times B)\geq \dim(A)+\dim(B)$ %holds for any measurable sets 
%$A,B$ and that
%linearly dependent vectors extensions of $\underline{\xi}_{n}$ 
%in $\mathbb{R}^{m+n}$
%are of lower dimension and thus do not violate \eqref{eq:chrom}.
Clearly, the estimate
\eqref{eq:chrom} can alternatively derived from \eqref{eq:jarn}.
We calculate
\[
\mathscr{D}_{k}\Big(\frac{n\lambda+n-1}{n(k-1)(\lambda+1)+1}\Big) 
-(\mathscr{D}_{n}(\lambda)+k-n)=\frac{(\lambda n-k+n-1)(kn-1)}{(1+\lambda)nk},
\]
the right hand side is non-negative as soon 
as $\lambda\geq (k-n+1)/n$. We derive that there is equality in \eqref{eq:chrom} precisely
for $\lambda=(k-n+1)/n$ to obtain
$\mathscr{D}_{k}(\frac{1}{k})= k$. Hence, for larger $\lambda$,
from a metrical point of view, 
the majority of vectors contributing to the left set of \eqref{eq:chrom}
is not coming from $\lambda$-approximable points
in a projection to $n$ coordinates.
We next establish corresponding going-up results concerning the uniform exponents. 

\begin{theorem}  \label{schlechtsatz}
	Keep the definitions and assumptions of Theorem~\ref{allgemein}. If we assume
	that 
	\begin{equation} \label{eq:kondi}
	\widehat{\lambda}_{n}(\underline{\xi}_{n})>\frac{k-n+1}{n},
	\end{equation}
	then 
	\begin{equation}  \label{eq:schwachmaten}
	\widehat{\lambda}_{k}(\underline{\xi}_{k})\geq 
	\frac{\widehat{\lambda}_{n}(\underline{\xi}_{n})+n-2}{(n-1)(k-1)}>\frac{1}{k}.
	\end{equation}
\end{theorem}

Theorem~\ref{schlechtsatz} is of no interest for Veronese curves
as condition \eqref{eq:kondi} contradicts \eqref{eq:rlsc} as 
soon as $k>n$.
The spectrum of $\widehat{\lambda}_{N}$ among
$\underline{\xi}\in\mathbb{R}^{N}$ that 
are  $\mathbb{Q}$-linearly independent with $\{ 1\}$
equals $[1/N,1]$, as follows for example from the constructions 
in~\cite[Theorem~2.5]{j1}, or alternatively Roy's deep existence
result~\cite{roy3}.
Consequently the condition \eqref{eq:kondi} can be satisfied for $2\leq n\leq k\leq 2n-2$.  
Metrical implications in the spirit of~\eqref{eq:chrom} 
between sets
\[
\widehat{\mathscr{D}}_{N}(\lambda):= \dim \{ \underline{\xi}\in\mathbb{R}^{N}: \widehat{\lambda}_{N}(\underline{\xi})\geq \lambda\}, \qquad\qquad \lambda\in[1/N,1],
\]
in various dimensions $N$ follow, we omit explicitly stating them. 
If $N=1$, then $\widehat{\lambda}_{1}(\xi)=1$ for
any irrational $\xi$, see~\cite{kredi}.
For larger $N$, the problem 
of determining $\widehat{\mathscr{D}}_{N}(\lambda)$
is only solved in a paper in preparation for
$N=2$ by Das, Fishman, Simmons, Urba\'{n}ski~\cite{ank},~\cite{ank2} and independently
by Bugeaud, Cheung, Chevallier~\cite{bcc}. 
However, when taking $n=2$ Theorem~\ref{schlechtsatz} 
does not provide new information on
any value $\widehat{\mathscr{D}}_{N}(\lambda)$. 

We believe that apart from obvious obstructions, the restrictions \eqref{eq:schwachmat}, \eqref{eq:schwachmaten} on sequences
are sufficient for the projections of suitable $\underline{\xi}\in \mathbb{R}^{\mathbb{N}}$ 
to attain all values simultaneously.

\begin{conjecture}  \label{kondsch}
	Let $(\lambda_{N})_{N\geq 1}$ and $(\widehat{\lambda}_{N})_{N\geq 1}$
	be non-increasing sequences of reals
	satisfying $\widehat{\lambda}_{N}\geq 1/N$ for $N\geq 1$, the estimates
	\begin{equation}  \label{eq:nonedtor}
	\widehat{\lambda}_{N}+
	\frac{\widehat{\lambda}_{N}^{2}}{\lambda_{N}}+\cdots+
	\frac{\widehat{\lambda}_{N}^{N}}{\lambda_{N}^{N-1}} \leq 1, 
	\qquad\qquad N\geq 1,
	\end{equation}
	originating in~\cite{mamo} 
	and for all $k\geq n\geq 1$ the relations
	\[
	\lambda_{k}\geq 
	\frac{n\lambda_{n}+n-1}{n(k-1)(\lambda_{n}+1)+1}, \qquad
	\widehat{\lambda}_{k}\geq 
	\frac{n+\widehat{\lambda}_{n}-2}{(n-1)(k-1)}.
	\]
    Then there is $\underline{\xi}\in \mathbb{R}^{\mathbb{N}}$ such that $\lambda_{N}(\underline{\xi}_{N})=\lambda_{N}$ and
	$\widehat{\lambda}_{N}(\underline{\xi}_{N})=\widehat{\lambda}_{N}$ 
	for all $N\geq 1$.
\end{conjecture}

This resembles the ''main problem'' formulated in~\cite[Section~3.4]{bugbuch} regarding approximation to the Veronese curve, which however
involves different types of exponents. Less audacious conjectures can be readily stated by
considering only one type of exponents, i.e. either ordinary or uniform.
We omit the formulation.

We close with a version of Theorem~\ref{t2} for the 
$\mathbb{Q}$-linearly independent case, that is again considerably weaker
but admits an easy deduction from classical transference principles.

\begin{theorem}  \label{consid}
	Upon the assumptions of Theorem~\ref{allgemein}, we have
	\[
	\lambda_{k}(\underline{\xi}_{k}) \geq \frac{(\widehat{w}_{n}(\underline{\xi}_{n})-1)w_{n}(\underline{\xi}_{n})}{((k-2)\widehat{w}_{n}(\underline{\xi}_{n})+1)w_{n}+(k-1)\widehat{w}_{n}(\underline{\xi}_{n})}.
	\]
\end{theorem}

% nochmal anschauen und mit 2dim formel das/simmons/urbanski 
% vergleichen!!!!!!!!!!!!!!!!!!!!!!!!!!!!

\section{Parametric geometry of numbers and preliminary results} \label{paramet}

Our proofs are based on classical tools
from geometry of numbers, in particular Minkowski's Convex Body Theorems.
To simplify to some extent the slightly cumbersome calculations that appear,
we work within the framework of parametric geometry of numbers
introduced by Schmidt and Summerer in~\cite{ss}. 
We slightly deviate from its original notation and put emphasis
on the concrete estimates we require. We
refer to~\cite{ss,ssch} for a more comprehensive introduction, see also Roy~\cite{roy3} for a different setup. Recall 
that the $j$-th successive minimum of a convex body $K$ with respect
to a lattice $\Lambda$ is the minimum $\lambda>0$ so that $\lambda K$
contains $j$ linearly independent points of $\Lambda$.

\subsection{Parametric functions}
Let $N\geq 1$ an integer and 
$\underline{\xi}\in\mathbb{R}^{N}$ be given.
Let $q>0$ be a parameter and let $Q=e^q$. Define convex bodies
\[
K(Q)=\{ (z_{0},\ldots,z_{N}): \vert z_{0}\vert \leq Q,\quad \vert z_{1}\vert \leq Q^{-1/N},\ldots,
\quad \vert z_{N}\vert \leq Q^{-1/N}\},
\]
and a lattice by
\[
\Lambda_{\underline{\xi}}=\{ (x,\xi_{1}x-y_{1}, \ldots,\xi_{N}x-y_{N}): \; x,y_{j}\in\mathbb{Z}\}.
\]
The successive minima of $K(Q)$ with 
respect to $\Lambda_{\underline{\xi}}$
contain important information on simultaneous rational
approximation to $(\xi_{1},\ldots,\xi_{N})$.
For $1\leq j\leq N+1$,
denote by $\tau_{N,j}(Q)$ the $j$-th
successive minimum and derive $\psi_{N,j}(Q)$ and $L_{N,j}(q)$ as in~\cite{ss} via
\[
\psi_{N,j}(Q)=\frac{\log \tau_{N,j}(Q)}{q}, \qquad\quad L_{N,j}(q)=\log \tau_{N,j}(Q)=q\psi_{N,j}(Q).
\]
The functions $L_{N,j}$ are piecewise linear with 
slopes among $\{-1,1/N\}$, see~\cite{ss}.

The linear form problem corresponds to dual approximation problem, i.e.
the successive minima problem with
respect to the dual parametric convex bodies
\[
K^{\ast}(Q)=\{ \underline{y}\in\mathbb{R}^{N+1}: 
\vert\underline{y}\cdot \underline{z}\vert\leq 1, \underline{z}\in K(Q)\}
\] 
given in coordinates by
\[
K^{\ast}(Q)=\{ (y_{0},\ldots,y_{N})\in\mathbb{R}^{N+1}: Q\vert y_{0}\vert+
Q^{-N}\vert y_{1}\vert+\cdots+Q^{-N}\vert y_{N}\vert \leq 1\},
\]
%     nachprufen ob korrekter dualkoerper!!!!!!!!!!!!!!!!!!
and the dual lattice $\Lambda_{\underline{\xi}}^{\ast}=\{ \underline{y}\in\mathbb{R}^{N+1}: 
\underline{y}\cdot \underline{z}\in\mathbb{Z}, \underline{z}\in \Lambda_{\xi}\}$,
given as
\[
\Lambda_{\underline{\xi}}^{\ast}=
\{ (x_{0}+\xi_{1}x_{1}+\cdots+\xi_{N}x_{N}, x_{1},\ldots,x_{N})\in\mathbb{R}^{N+1}: x_{j}\in\mathbb{Z}\}.
\]
Again, for $1\leq j\leq N+1$,  from successive minima 
with respect to $K^{\ast}(Q)$
and $\Lambda_{\underline{\xi}}^{\ast}$ we derive functions $\psi_{N,j}^{\ast}(Q)$
and $L_{N,j}^{\ast}(q)$ accordingly.
Any $L_{N,j}^{\ast}(q)$ is locally induced by the 
function $L_{N,\underline{x}}^{\ast}(q)$ for some
$\underline{x}=(x_{0},x_{1},\ldots,x_{N})\in \mathbb{Z}^{N+1}$
defined as
\begin{equation} \label{eq:defin}
L_{N,\underline{x}}^{\ast}(q)= \max\left\{ \log \Vert \underline{x}\Vert_{\infty}-\frac{q}{N},\; \log \scp{\underline{x},\underline{\xi}}_{N}+q \right\},
\end{equation}
where
\[
\Vert \underline{x}\Vert_{\infty}= \max_{0\leq i\leq N} \vert x_{i}\vert, 
\qquad\qquad \scp{\underline{x},\underline{\xi}}_{N}= \vert x_{0}+\xi_{1} x_{1}+\cdots+\xi_{N}x_{N}\vert.
\]
The functions $L_{N,j}^{\ast}(q)$ therefore have slope among $\{1,-1/N\}$.
For $j=1$, the value $L_{N,1}^{\ast}(q)$ just equals
the minimum of $L_{N,\underline{x}}^{\ast}(q)$
over $\underline{x}\in\mathbb{Z}^{N+1}\setminus \{ \underline{0}\}$. 
Also notice that for successive powers $\underline{\xi}=(\xi,\xi^2,\ldots,\xi^N)$ the scalar product 
$\scp{\underline{x},\underline{\xi}}_{N}$ may be written $|P(\xi)|$ with $P\in \mathbb{Z}[T]$ of degree at most $N$. 
We close this section by defining the upper and lower limits
\[
\underline{\psi}_{N,j}=\liminf_{Q\to\infty} \psi_{N,j}(Q),
\qquad\quad \overline{\psi}_{N,j}=\limsup_{Q\to\infty} \psi_{N,j}(Q),
\]
and $\underline{\psi}_{N,j}^{\ast}, \overline{\psi}_{N,j}^{\ast}$ accordingly
%
%\[
%\underline{\psi}_{N,j}^{\ast}=\liminf_{Q\to\infty} \psi_{N,j}^{\ast}(Q),
%\qquad \overline{\psi}_{N,j}^{\ast}=\limsup_{Q\to\infty} \psi_{N,j}^{\ast}(Q).
%\]
%
that are linked to classical exponents, see next section.

\subsection{Minkowski's Theorems, Mahler's duality, relation to classical exponents}  \label{s51}

Variants of Dirichlet's Theorem, 
or Minkowski's First Convex Body Theorem, imply 
$\psi_{N,1}(Q)<0$ and
$L_{N,1}(q)<0$, as well as $\psi_{N,1}^{\ast}(Q)<0$ and $L_{N,1}^{\ast}(q)<0$, for all $q>0$.
Minkowski's Second Convex Body Theorem yields
\begin{equation} \label{eq:lll}
\left\vert \sum_{j=1}^{N+1} \psi_{N,j}(Q)\right\vert \leq \frac{C_{N}}{q},\qquad
\left\vert \sum_{j=1}^{N+1} L_{N,j}(q)\right\vert \leq C_{N}, 
\qquad\qquad  q>0,
\end{equation}
and similarly
\begin{equation} \label{eq:rrr}
\left\vert \sum_{j=1}^{N+1} \psi_{N,j}^{\ast}(Q)\right\vert \leq \frac{C_{N}^{\ast}}{q},\qquad
\left\vert \sum_{j=1}^{N+1} L_{N,j}^{\ast}(q)\right\vert \leq C_{N}^{\ast}, 
\qquad\qquad  q>0,
\end{equation}
for constants $C_{N}>0$ and $C_{N}^{\ast}>0$.  

Our two approximation problems, simultaneous approximation 
and linear forms, are connected by
Mahler's theorem on dual convex bodies. It implies 
\begin{equation} \label{eq:mahler}
\vert\psi_{N,1}(Q)+\psi_{N,N+1}^{\ast}(Q)\vert \leq \frac{c_{N}}{q}, \qquad\qquad \vert\psi_{N,1}^{\ast}(Q)+\psi_{N,N+1}(Q)\vert \leq \frac{c_{N}}{q},
\end{equation}
for some constant $c_{N}>0$ independent from $Q$.
In particular
\begin{equation} \label{eq:jaja}
\underline{\psi}_{N,1}=-\overline{\psi}_{N,N+1}^{\ast}, \qquad\qquad 
\overline{\psi}_{N,1}=-\underline{\psi}_{N,N+1}^{\ast}.
\end{equation}
From \eqref{eq:rrr} and \eqref{eq:mahler} we obtain
\begin{equation} \label{eq:schranke}
\sum_{j=1}^{N} \psi_{N,j}^{\ast}(Q)= \psi_{N,1}(Q)+O(q^{-1}),
\qquad \sum_{j=1}^{N} \psi_{N,j}(Q)= \psi_{N,1}^{\ast}(Q)+O(q^{-1}).
\end{equation}
From \eqref{eq:rrr} one may readily derive~\cite[(1.11)]{ssch}, which reads
in our notation
\begin{equation} \label{eq:thatp}
j\underline{\psi}_{N,j}+ (N+1-j)\overline{\psi}_{N,N+1} \geq 0, \qquad
j\overline{\psi}_{N,j}+ (N+1-j)\underline{\psi}_{N,N+1} \geq 0
\end{equation}
and similarly for $\psi_{N,j}^{\ast}$.
%
%\begin{equation} \label{eq:dsch}
%j\underline{\psi}_{N,j}^{\ast}+ (N+1-j)\overline{\psi}_{N,N+1}^{\ast} \geq 0, %\qquad
%j\overline{\psi}_{N,j}^{\ast}+ (N+1-j)^{\ast}\underline{\psi}_{N,N+1} \geq 0
%\end{equation}
%
For $j=1$ we immediately deduce~\cite[(1.11)]{ssch} that may be written
\begin{equation} \label{eq:duales}
-\overline{\psi}_{N,N+1}^{\ast}(Q)\leq \frac{1}{N}\cdot  \underline{\psi}_{N,1}^{\ast}(Q), \qquad\qquad
-\overline{\psi}_{N,N+1}(Q)\leq \frac{1}{N}\cdot 
\underline{\psi}_{N,1}(Q).
\end{equation}
In fact only the right estimates occur in~\cite{ssch}, 
but the dual left inequalities admit an analogous proof.

In~\cite[Theorem~1.4]{ss}, a fundamental link between 
the upper and lower limits on one side
and the exponents from Section~\ref{introd} on the other side
is given via the identities
\begin{equation} \label{eq:umrechnen}
(1+\lambda_{N}(\underline{\xi}))(1+\underline{\psi}_{N,1})=
(1+\widehat{\lambda}_{k}(\underline{\xi}))(1+\overline{\psi}_{N,1})=\frac{N+1}{N}, 
\end{equation}
and 
\begin{equation} \label{eq:umrechnen2}
(1+w_{N}(\underline{\xi}))\Big(\frac{1}{N}+\underline{\psi}_{N,1}^{\ast}\Big)=
(1+\widehat{w}_{N}(\underline{\xi}))\Big(\frac{1}{N}+\overline{\psi}_{N,1}^{\ast}\Big)=\frac{N+1}{N}.
\end{equation}
In fact we will often implicitly use parametric versions of \eqref{eq:umrechnen}, \eqref{eq:umrechnen2}, stating that 
for any $1\leq j\leq N+1$, a set of $j$ linearly independent 
integer vectors inducing 
an exponent $\lambda$ resp. $w$ in \eqref{eq:lambdar} resp. \eqref{eq:lammda} gives rise to $q$ with the according identity linking $\lambda$ with $\psi_{N,j}(q)$ resp. $w$ with $\psi_{N,j}^{\ast}(q)$.

\subsection{A transference lemma and an observation on minimal polynomials}
The following lemma stems from 
a simple calculation and will be frequently applied throughout our proofs. It 
describes the transformation of the 
functions $L_{n,\underline{x}}^{\ast}$ above
induced by some $\underline{x}=(x_{0},\ldots,x_{n})\in\mathbb{Z}^{n+1}$,
into $L_{k,\underline{x}^{\prime}}^{\ast}$ in some larger dimension $k>n$ upon setting $\underline{x}^{\prime}=(x_{0},\ldots,x_{n},0,\ldots,0)\in \mathbb{Z}^{k+1}$. 
In the case of successive powers 
we easily gain some improvement by varying $\underline{x}^{\prime}$
that turns out crucial.

\begin{lemma} \label{lemuren}
	Let $k\geq n\geq 1$ be integers. Further let
	$\underline{\xi}=(\xi_{1},\ldots,\xi_{k})$ be a real
	vector 
	and $\tilde{\underline{\xi}}=(\xi_{1},\ldots,\xi_{n})$
	the restriction of $\underline{\xi}$ to the first $n$ components.
	Assume 
	$\underline{x}=(x_{0},x_{1},\ldots,x_{n})\in \mathbb{Z}^{n+1}$ 
	and $q>0$ and $\psi$ are parameters so that the function
	$L_{n,\underline{x}}^{\ast}$ associated to 
	$\tilde{\underline{\xi}}$ and $\underline{x}$
	satisfies
	\[
	L_{n,\underline{x}}^{\ast}(q) \leq \psi q.
	\]
	Let
	\begin{equation} \label{eq:os}
	q^{\prime}= q\frac{(n+1)k}{n(k+1)}, \qquad\qquad 
	\psi^{\prime}=\Phi_{k,n}(\psi)
	\end{equation}
	where $\Phi_{k,n}$ is the affine function given as
	\begin{equation}  \label{eq:phink}
	\Phi_{k,n}(t):=(t-1)\frac{(k+1)n}{k(n+1)}+1=\frac{n(k+1)}{(n+1)k}t+ \frac{k-n}{k(n+1)}.
	\end{equation}
	Then for
	\begin{equation} \label{eq:ous}
	\underline{x}^{\prime}=(x_{0},x_{1},\ldots,x_{n},0,0,\ldots,0)\in \mathbb{Z}^{k+1},
	\end{equation}
	we have
	\[
	L_{k,\underline{x}^{\prime}}^{\ast}(q^{\prime}) \leq \psi^{\prime}q^{\prime}.
	\]
	Moreover, if $\xi_{j}=\xi^{j}$ for $1\leq j\leq n$ and 
	some $\xi\in(0,1)$, then 
	the same claim holds for any 
	vector $\underline{x}^{\prime}=\underline{x}^{\prime}_{i}$
	of the form 
	\begin{equation} \label{eq:ous2}
	\underline{x}^{\prime}_{i}=(0,\ldots,0,x_{0},x_{1},\ldots,x_{n},0,0,\ldots,0)\in \mathbb{Z}^{k+1}, \qquad\qquad  1\leq i\leq k-n+1,
	\end{equation}
	where in $\underline{x}^{\prime}_{i}$ the 
	coordinate $x_{0}$ is in position $i$. 
\end{lemma}

\begin{proof}
	First we treat the case of general vectors 
	$\underline{\xi}$.
	Observe that obviously for $\underline{x}^{\prime}$ as in
	\eqref{eq:ous} we have
	\[
	\Vert \underline{x}\Vert_{\infty}=\Vert \underline{x}^{\prime}\Vert_{\infty}, \qquad
	\scp{\underline{x}^{\prime},\underline{\xi}}_{k}= \scp{\underline{x},\tilde{\underline{\xi}}}_{n}.
	\]
	Hence, according to \eqref{eq:defin} we have
	\[
	L_{n,\underline{x}}^{\ast}(q)= \max \left\{ \log \Vert \underline{x}\Vert_{\infty}-\frac{q}{n},\; \log \scp{\underline{x},\tilde{\underline{\xi}}}_{n}+q \right\}
	\]
	and
	\[
	L_{k,\underline{x}^{\prime}}^{\ast}(q^{\prime})= \max \left\{ \log \Vert \underline{x}\Vert_{\infty}-\frac{q^{\prime}}{k},\; \log \scp{\underline{x},\tilde{\underline{\xi}}}_{n}+q^{\prime} \right\}.
	\]
	Thus it suffices to check that for $q^{\prime}, \psi^{\prime}$ as given 
	in \eqref{eq:os}, the inequalities
	\[
	\log \Vert \underline{x}\Vert_{\infty}-\frac{q}{n}\leq q\psi, \qquad
	\log \scp{\underline{x},\tilde{\underline{\xi}}}_{n}+q\leq q\psi
	\]
	imply
	\[
	\log \Vert \underline{x}\Vert_{\infty}-\frac{q^{\prime}}{k}\leq q^{\prime}\psi^{\prime}, \qquad \log \scp{\underline{x},\tilde{\underline{\xi}}}_{n}+q^{\prime}\leq q^{\prime}\psi^{\prime}.
	\]
	We leave these elementary calculations to the reader.
	
	Now take the special case $\xi_{j}=\xi^{j}$ for $1\leq j\leq n$ and 
	some $\xi\in(0,1)$. Then 
	if we identify $\underline{x}$ with the polynomial 
	$P(T)=x_{0}+x_{1}T+\cdots+x_{n}T^{n}$, we readily check that
	a right shift of $\underline{x}$ within $\underline{x}^{\prime}$ 
	corresponds
	to a multiplictation by $T$, so that
	$\underline{x}^{\prime}_{i}$ corresponds to $T^{i-1}P(T)$ for
	$1\leq i\leq k-n+1$. 
	Since $\xi\in(0,1)$
	we have $\vert\xi^{j}P(\xi)\vert\leq \vert P(\xi)\vert$
	for $j\geq 0$, and thus again
	\[
	\Vert \underline{x}\Vert_{\infty}=
	\Vert \underline{x}^{\prime}_{i}\Vert_{\infty}, \qquad
	\scp{\underline{x}^{\prime}_{i},\underline{\xi}}_{k}\leq \scp{\underline{x},\tilde{\underline{\xi}}}_{n},
	\]
	for any $1\leq i\leq k-n+1$.
	The claim follows as above.
\end{proof}

We finish this section with a proposition that extends an observation of Wirsing~\cite[Hilfssatz~4]{wirsing}. It concerns the degrees 
of well approximating polynomials that play a role
in the proofs below. It is unrelated to
parametric geometry of numbers. 

\begin{proposition}  \label{propper}
	Let $\xi$ be a transcendental real number, $n\geq 2$ an integer and $\epsilon>0$. 
	Then
	\begin{equation}  \label{eq:infoft}
	|P(\xi)| < H(P)^{-w_{n}(\xi)+\epsilon  }
	\end{equation}
	has infinitely many solutions in irreducible integer
	polynomials $P$ of arbitrarily large height and degree at 
	least $\lceil\widehat{w}_{n}(\xi)\rceil-n+1$ and at most $n$.
	On the other hand, the inequality
	\begin{equation}  \label{eq:infoft2}
	|P(\xi)| < H(P)^{-\widehat{w}_{n}(\xi)+\epsilon  }
	\end{equation}
	only finitely many solutions in integer polynomials of degree at most $\lceil\widehat{w}_{n}(\xi)\rceil-n$ if $\epsilon$ is small enough.
	Moreover, if $w_{n}(\xi)>w_{n-1}(\xi)$ and $\epsilon$ is small enough, then there exist
	$P$ irreducible of degree $n$ satisfying \eqref{eq:infoft}
	of arbitrarily large height. 
\end{proposition}

\begin{proof} 
	By~\cite[Hilfssatz~4]{wirsing}, we may
	choose irreducible integer polynomials $P$ of degree at most $n$ with property \eqref{eq:infoft} of arbitrarily large height. 
	The last, conditional claim follows immediately when taking $\epsilon=(w_{n}(\xi)-w_{n-1}(\xi))/2$ as then these polynomials
	cannot have degree smaller than $n$. For the other claims, 
	we conclude by showing that the degree
	of polynomials $P$ satisfying the weaker property \eqref{eq:infoft2} can be $\lceil\widehat{w}_{n}(\xi)\rceil-n$ or less only for finitely many $P$.
	
	So let $m\in\{1,2,\ldots, n\}$ be the minimum integer so that \eqref{eq:infoft2} has
	infinitely many solutions in integer polynomials $P$ of degree
	$m$ or less. 
	It was shown in~\cite[Theorem~2.3]{buschlei} that for any
	transcendental real $\xi$ and any integers $m,n\geq 1$ we have
	\begin{equation} \label{eq:fromm}
	\min\{ w_{m}(\xi), \widehat{w}_{n}(\xi) \} \leq m+n-1.
	\end{equation}
	Assume contrary to our claim that $m\leq \lceil\widehat{w}_{n}(\xi)\rceil-n$. Then
	$\widehat{w}_{n}(\xi)> \lceil\widehat{w}_{n}(\xi)\rceil -1 \geq m+n-1$, and from \eqref{eq:fromm} we conclude
	$w_{m}(\xi)\leq m+n-1$. On the other hand by
	definition of $m$ we have  
	$w_{m}(\xi)\geq \widehat{w}_{n}(\xi)$. Combining we get the contradiction 
	\[
	\widehat{w}_{n}(\xi)\leq w_{m}(\xi) \leq m+n-1 < \widehat{w}_{n}(\xi).
	\]
	Hence indeed $m\geq \lceil\widehat{w}_{n}(\xi)\rceil-n+1$.
\end{proof}

\section{Proofs of the mixed properties}  \label{6}

We first prove the results of Section~\ref{mix} as the proofs
are a bit easier.
For simplicity and improved readability, we will omit
the argument $\xi$ in the exponents $w_{.}, \widehat{w}_{.}, \lambda_{.}, \widehat{\lambda}_{.}$ in all proofs. Moreover, it will be throughout
understood that $\epsilon_{i}$ derived from some initial $\epsilon>0$
are positive and tend to $0$ as $\epsilon$ does.

\subsection{Proof of Theorem~\ref{t2}}

Consider the combined graph of the linear form problem with
respect to $(\xi,\xi^{2},\ldots,\xi^{n})$. 
Let $\epsilon>0$.
By \eqref{eq:umrechnen2}, at
certain arbitrarily large $Q=e^{q}$ the first minimum satisfies
\[
|\psi_{n,1}^{\ast}(Q)-\frac{n-w_{n}}{n(1+w_{n})}|=
|\frac{L_{n,1}^{\ast}(q)}{q}-\frac{n-w_{n}}{n(1+w_{n})}|< \epsilon.
\]
Let 
\begin{equation} \label{eq:aal}
\alpha^{\ast}= \frac{n-w_{n}}{n(1+w_{n})}.
\end{equation}
We may assume that $q$ is a local minimum of $L_{n,1}^{\ast}$.
Let $s^{\ast}>0$ be the smallest positive number such that
$L_{n,1}^{\ast}(q+s^{\ast})=L_{n,2}^{\ast}(q+s^{\ast})$, so that $q+s^{\ast}$ is
the first meeting point of first and second minimum functions
to the right of $q$. Let $S^{\ast}=e^{s^{\ast}}$ and $Q^{\ast}=QS^{\ast}=e^{q+s^{\ast}}$. Then by \eqref{eq:umrechnen2}  
 we have
\begin{equation} \label{eq:rindver}
\psi_{n,2}^{\ast}(Q^{\ast})=\frac{L_{n,2}^{\ast}(q+s^{\ast})}{q+s^{\ast}} \leq 
\frac{n-\widehat{w}_{n}}{n(1+\widehat{w}_{n})}+\epsilon_{1}.
\end{equation}
Let
\[
\beta^{\ast}:= \frac{n-\widehat{w}_{n}}{n(1+\widehat{w}_{n})}.
\]
Since every local maximum of $L_{n,1}^{\ast}$ is a local minimum 
of $L_{n,2}^{\ast}$, the function $L_{n,1}^{\ast}$ increases with slope
$+1$ in the interval $[q,q+s^{\ast}]$. Thus
we have $L_{n,2}^{\ast}(q+s^{\ast})=L_{n,1}^{\ast}(q+s^{\ast})=L_{n,1}^{\ast}(q)+s^{\ast}$ and
we calculate
\[
\frac{s^{\ast}}{q}= \frac{L_{n,2}^{\ast}(q+s^{\ast})}{q} - \frac{L_{n,1}^{\ast}(q)}{q}\leq 
(\beta^{\ast}+\epsilon_{1})\frac{q+s^{\ast}}{q} - \alpha^{\ast} +\epsilon=
\beta^{\ast}-\alpha^{\ast}+ \beta^{\ast} \frac{s^{\ast}}{q}+\epsilon_{2}
\]
%   mach vaeps zu eps_3   etc
and solving for $s^{\ast}/q$ thus
\[
\frac{s^{\ast}}{q}
\leq \frac{\beta^{\ast}-\alpha^{\ast}}{1-\beta^{\ast}}+ \epsilon_{3}.
\]
Since $L_{n,2}^{\ast}$ has slope at least $-1/n$, with \eqref{eq:rindver} and
inserting for $\alpha^{\ast}$ and $\beta^{\ast}$ at once, we infer
\begin{align*}
\psi_{n,2}^{\ast}(Q)&=\frac{L_{n,2}^{\ast}(q)}{q}\leq \frac{1}{q}(L_{n,2}^{\ast}(q+s^{\ast})+\frac{s^{\ast}}{n})
= \frac{ q+s^{\ast} }{q} \frac{ L_{n,2}^{\ast}(q+s^{\ast}) }{q+s^{\ast}}
+ \frac{1}{n}\frac{s^{\ast}}{q}   \\
&\leq \left(1+\frac{\beta^{\ast}-\alpha^{\ast}}{1-\beta^{\ast}} \right)\beta^{\ast} + 
\frac{1}{n}\frac{ \beta^{\ast}-\alpha^{\ast} }{1-\beta^{\ast} } + \epsilon_{4}
\leq \frac{\widehat{w}_{n}(n-w_{n})+(n+1)(w_{n}-\widehat{w}_{n})}{n\widehat{w}_{n}(1+w_{n})} + \epsilon_{5}.
\end{align*}
For simplicity let
\begin{equation} \label{eq:gga}
\gamma^{\ast}= \frac{\widehat{w}_{n}(n-w_{n})+(n+1)(w_{n}-\widehat{w}_{n})}{n\widehat{w}_{n}(1+w_{n})}.
\end{equation}
Now we transition to dimension $k$.
Let  $\underline{x}_{1}, \underline{x}_{2}$ be the integer
points inducing $L_{n,1}^{\ast}(q), L_{n,2}^{\ast}(q)$
 according to \eqref{eq:defin} for our $q$ above, respectively.
We will implictily 
identify $\underline{x}_{j}=(x_{j,0},\ldots,x_{j,n})$ 
with polynomials $P_{j}(T)=x_{j,0}+ x_{j,1}T+ \cdots + x_{j,n}T^{n}$, 
for $j=1,2$. 
Say $d$ is the exact degree of $P_{1}$, where
$d\in\{1,2,\ldots,n\}$. Consider the set of
$k-n+2$ polynomials
\[
\mathscr{R}=\{ R_{1},\ldots,R_{k-n+2} \}
=\{ P_{1}, T^{n-d+1}P_{1},T^{n-d+2},\ldots,T^{k-d}P_{1},P_{2}\}.
\]
It consists of polynomials of degree at most $k$ and 
we readily check $\mathscr{R}$ is linearly independent.
Indeed $P_{1}, P_{2}$ are linearly independent and adding
one by one the remaining polynomials from
$T^{n-d+1}P_{1}$ up to $T^{k-d}P_{1}$ increases the dimension
in each step because the new polynomial has larger degree 
than any polynomial that occurred before.

Now, for $1\leq u\leq k-n+2$, the coefficient vector of $R_{u}$
can be interpreted as a vector $\underline{x}_{i}^{\prime}$ 
with some $i=i(u)$ as in \eqref{eq:ous2},
derived from putting $\underline{x}=\underline{x}_{1}$
if $1\leq u\leq k-n+1$ and $\underline{x}=\underline{x}_{2}$ 
if $u=k-n+2$. For simplicity denote $L_{k,R_{u}}^{\ast}(q)= L_{k,\underline{x}_{i}^{\prime}}^{\ast }(q)$ the functions 
in \eqref{eq:defin} upon this identification.  
Hence, with $\Phi_{k,n}$ from \eqref{eq:phink},
from Lemma~\ref{lemuren} 
we get that the first $k-n+1$ polynomials in $\mathscr{R}$
induce average slope 
$\Phi_{k,n}(\psi_{n,1}^{\ast}(Q))$ in $[0,q^{\prime}]$, and one more
induces average slope $\Phi_{k,n}(\psi_{n,2}^{\ast}(Q))$ in $[0,q^{\prime}]$, at some transformed
position $q^{\prime}=(n+1)k/(n(k+1))\cdot q$. Writing 
$Q^{\prime}=e^{q \prime}$, in other words we establish
\[
\psi_{k,k-n+1}^{\ast}(Q^{\prime})=\frac{L_{k,k-n+1}^{\ast}(q^{\prime})}{q^{\prime}}\leq \min_{1\leq u\leq k-n+1} \frac{L_{k,R_{u}}^{\ast}(q^{\prime})}{q^{\prime} }
\leq \Phi_{k,n}(\psi_{n,1}^{\ast}(Q))
\]
and
\[
 \psi_{k,k-n+2}^{\ast}(Q^{\prime})=\frac{L_{k,k-n+1}^{\ast}(q^{\prime})}{q^{\prime}}\leq  \frac{L_{k,R_{k-n+2}}^{\ast}(q^{\prime})}{q^{\prime} } \leq \Phi_{k,n}(\psi_{n,2}^{\ast}(Q)).
\]
Then $k+1-\vert \mathscr{R}\vert=k+1-(k-n+2)=n-1$ successive minima functions remain.
Thus from \eqref{eq:rrr} for the last function at $Q^{\prime}$ we derive 
\[
\psi_{k,k+1}^{\ast}(Q^{\prime})=\frac{L_{k,k+1}^{\ast}(q^{\prime})}{q^{\prime}}
\geq -\frac{(k-n+1)\Phi_{k,n}(\psi_{n,1}^{\ast}(Q))+\Phi_{k,n}(\psi_{n,2}^{\ast}(Q))}{n-1}-O(q^{\prime -1}), 
\]
thus 
\[
\psi_{k,k+1}^{\ast}(Q^{\prime})\geq -\frac{(k-n+1)\Phi_{k,n}(\alpha^{\ast})+\Phi_{k,n}(\gamma^{\ast})}{n-1}-\epsilon_{6}-O(q^{\prime -1}).
\]
From Mahler's duality \eqref{eq:mahler}, for $\psi_{k,1}(Q^{\prime})$
the average slope in the successive minima diagram of
the first successsive minimum in $[0,q^{\prime}]$ (with respect to $(\xi,\xi^{2},\ldots,\xi^{k})$) we obtain
\begin{equation} \label{eq:instead}
\psi_{k,1}(Q^{\prime})\leq -\psi_{k,k+1}^{\ast}(Q^{\prime})+O(q^{\prime -1})\leq 
\frac{(k-n+1)\Phi_{k,n}(\alpha^{\ast})+\Phi_{k,n}(\gamma^{\ast})}{n-1}+\epsilon_{6}+O(q^{\prime -1}).
\end{equation}

Let $\epsilon>0$. As $Q\to\infty$,
with \eqref{eq:umrechnen} for $N=k, j=1$ applied to our estimate \eqref{eq:instead}, and inserting for $\alpha^{\ast}, \gamma^{\ast}$
from \eqref{eq:aal}, \eqref{eq:gga},
after a lengthy computation
we get a lower bound of the form
\[
\lambda_{k}\geq \frac{w_{n}\widehat{w}_{n}-
w_{n}+(n-k)\widehat{w}_{n}}{(n-2)w_{n}\widehat{w}_{n}+w_{n}
+(k-1)\widehat{w}_{n}}-\epsilon_{7}.
\]
Since $\epsilon_7$ can be arbitrarily close to $0$, the desired bound
is obtained. The proof is complete.

The key point of the proof was to find a relatively large
set $\mathscr{R}$ of linearly independent 
polynomials with small evaluation at $\xi$. For this we made extensive use
of the fact that we work with successive powers of some $\xi$. The proofs of Theorems~~\ref{th3},~\ref{gleich} rely on the same principle.

\subsection{Proof of Theorem~\ref{th3}}

To improve our result upon condition \eqref{eq:assu}, in the proof
of Theorem~\ref{th3} below
the main step is to notice that in this case
we can extend the polynomial set $\mathscr{R}$ from the proof
of Theorem~\ref{t2}
and still guarantee that it remains linearly independent.

We verify \eqref{eq:tarda} upon our assumption \eqref{eq:assu} and $k\leq 2n-2$.
Let
$\alpha^{\ast},\beta^{\ast},\gamma^{\ast}$ as in the proof of Theorem~\ref{t2}.
Further take the same $q$ and derived $q^{\prime}$.
In place of \eqref{eq:instead}, we show the stronger estimate
\begin{equation} \label{eq:i}
\psi_{k,1}(Q^{\prime})\leq \frac{(k-n+1)\Phi_{k,n}(\alpha^{\ast})+(k-n+1)\Phi_{k,n}(\gamma^{\ast})}{2n-1-k}+\epsilon_{4}+O(q^{\prime -1}).
\end{equation}
Observe that the denominator is positive by assumption.
By our hypothesis \eqref{eq:assu} and Proposition~\ref{propper},
we may assume that the polynomial $P_{1}$ inducing
$\psi_{n,1}^{\ast}(Q)\leq \alpha^{\ast}+o(1)$ 
is irreducible and of degree exactly $n$.
In particular coprime to the polynomial $P_{2}$ inducing  
$\psi_{n,2}^{\ast}(Q)\leq \gamma^{\ast}+o(1)$.
We claim that then the set of polynomials
\[
\tilde{\mathscr{R}}=\{ P_{1},TP_{1},\ldots,T^{k-n}P_{1}, P_{2}, TP_{2},\ldots,T^{k-n}P_{2}\}
\]
consists
of polynomials of degree at most $k$, and
is linearly independent. 
Indeed, otherwise if some non-trivial linear combination
within $\tilde{\mathscr{R}}$ 
vanishes identically, we have a polynomial identity 
\[
P_1(T)U(T)=P_2(T)V(T)
\]
with $U,V$ integer polynomials, 
$U$ of degree at most $k-n$ and $V$ of degree
at most $k-n\leq n-2<n$.
Thus $P_1$ has to divide
either $P_2$ or $V$. Clearly it cannot divide $V$ as $P_1$ has larger degree. 
However, it cannot divide $P_2$ either since $P_1$ is irreducible 
of degree $n$ and
$P_2$ has degree at most $n$ and is not a scalar multiple
of $P_1$.
We obtain a contradiction and our claim is proved.

From the above argument, 
in the $k$-dimensional combined graph,
with the same position $q^{\prime}$ as
in the proof of Theorem~\ref{t2},
we now have $k-n+1$ polynomials inducing average slope essentially
at most
$\Phi_{k,n}(\alpha^{\ast})$ in $[0,q^{\prime}]$, and further $k-n+1$ polynomials inducing average
slope essentially at most $\Phi_{k,n}(\gamma^{\ast})$ in $[0,q^{\prime}]$.
Thus
\[
\psi_{k,k-n+1}^{\ast}(Q^{\prime})\leq \Phi_{k,n}(\alpha^{\ast}),
\qquad \psi_{k,2(k-n+1)}^{\ast}(Q^{\prime})
\leq \Phi_{k,n}(\gamma^{\ast}).
\]
Then $k+1-\vert \tilde{\mathscr{R}}\vert=k+1-2(k-n+1)=2n-1-k\geq 1$
polynomials corresponding to successive minima remain.
Using Mahler's duality as in the proof of Theorem~\ref{t2},
this obviously implies \eqref{eq:i} as the sum of
$\psi_{k,j}(Q^{\prime})$ over $j=1,2,\ldots,k+1$ is $O(q^{\prime -1})$.
The rest of the proof is done analogously to Theorem~\ref{t2}, 
we skip the details and computation.

\begin{remark}
Considering $k=2n-1$, a similar argument 
implies $\Phi_{2n-1,n}(\alpha^{\ast})+\Phi_{2n-1,n}(\gamma^{\ast})\geq 0$,
for $\alpha^{\ast},\gamma^{\ast}$ in \eqref{eq:aal}, \eqref{eq:gga}, 
upon condition \eqref{eq:assu}. This turns
out to be equivalent to \eqref{eq:bush}, again upon the same hypothesis. 
Thereby we have found a new proof
of this fact that relies only on Minkowski's Second Convex
Body Theorem.
\end{remark} % FRAGE: neuer beweis liouville inequ. mit
% besseren konstanten moeglich???

 \subsection{Proof of Theorem~\ref{gleich}}  \label{glsec}

Gelfond's Lemma states that for polynomials $P,R$ of degree 
at most $N$
we have $H(PR) \asymp_{N} H(P)H(R)$. In particular for any integer $N$
there is some absolute $c(N)>0$ so 
that
\begin{equation}  \label{eq:gelfond}
H(PR) > c(N) \cdot H(P)H(R) \geq c(N)H(P)
\end{equation}
holds for all non-zero polynomials $P,R$ of degree at most $N$.
Using this property, the proof is similar to that of Theorem~\ref{th3} again.

     So let us prove Theorem~\ref{gleich} now.
	As recalled in Proposition~\ref{propper},
	inequality \eqref{eq:infoft}
	has solutions in irreducible integer polynomials $P$ of degree $\leq n$ and arbitrarily large height. 
	Let $P_{1}$ be such a polynomial. Let $c(n)$ as
	in \eqref{eq:gelfond} and put $M=(c(n)/2)\cdot H(P_{1})$. 
	By definition of $\widehat{w}_{n}(\xi)$
	there is an integer polynomial $P_{2}$ of degree at most $n$ with 
	\[
	H(P_{2}) \leq M, \qquad |P_{2}(\xi)|\leq M^{-\widehat{w}_{n}(\xi)+\epsilon/2}.
	\]
	By construction $P_{2}$ is not a multiple of $P_{1}$, 
	thus coprime with $P_{1}$.
	Hence we have found coprime $P_{1}, P_{2}$ with
	\begin{equation}  \label{eq:neuneu}
	\max_{i=1,2} H(P_{i}) \leq M, \qquad 
	\max_{i=1,2}  |P_{i}(\xi)|< M^{- \widehat{w}_{n}(\xi)+\epsilon}.
	\end{equation}
    Identitfy $P_{1}$ as above with its coefficient vector $\underline{x}\in\mathbb{Z}^{n+1}$
	and write $L_{n,P_{1}}^{\ast}(q)=L_{n,\underline{x}}^{\ast}(q)$ for the
	induced function from \eqref{eq:defin}, and similarly for $P_{2}$.
	Now by \eqref{eq:umrechnen2} with $N=n$, 
	estimates \eqref{eq:neuneu} induce parameters $Q=e^{q}$ with
	\begin{equation}  \label{eq:byeqab}
	\psi_{n,2}^{\ast}(Q) \leq  \max_{i=1,2} \frac{L_{n,P_{i} }^{\ast}(q)}{q} \leq  \frac{n+1}{n} \frac{1}{1+\widehat{w}_{n}(\xi)}- \frac{1}{n}+\epsilon_{1} = \frac{n-\widehat{w}_{n}(\xi)}{n(1+\widehat{w}_{n}(\xi))}+\epsilon_{1}.
	\end{equation}

	Let $d$ be the degree of $P_{1}$. Next we claim that
	\[
	\mathscr{R}:= \{ P_{1}(T), TP_{1}(T), \ldots, T^{k-d}P_{1}(T), P_{2}(T), TP_{2}(T), \ldots, T^{ \min\{ d-1, k-n\}} P_{2}(T)  \}
	\] 
	is a linearly independent set of integer polynomials of degree at 
	most $k$. Since $d\leq n$ and $\deg P_{2}\leq n$ as well, 
	only the linear independence needs to be checked.
	Indeed, otherwise there would again be a polynomial identity
	$P_{1}(T)U(T)= P_{2}(T)V(T)$
	with integer polynomials $U,V$ of degrees at most $k-d$ and 
	$d-1$ respectively, and a very similar argument as in the proof
	of Theorem~\ref{th3} shows this is impossible.
	%but since
	%$P$ is irreducible and does not divide $R$ it would have to divide 
	%$V$, which however is impossible as $V$ has smaller degree than $P$. 
	This proves the claim. 
	
	Since all polynomials in $\mathscr{R}$ also have height $\leq M$
	and evaluation at $\xi$ of absolute value smaller than $M^{- \widehat{w}_{n}(\xi)+\epsilon}$ if we assume $\xi\in(0,1)$, we have found
	\[
	h:= |\mathscr{R}|=(k-d+1) + (\min\{ d-1, k-n\}+1) = \min\{ k+1, 2k+2-d-n\}
	\]
	linearly independent integer
	polynomials $R_{1}, \ldots, R_{h}$ of degree at most $k$ and with 
	\[
	\max_{1\leq i\leq h} H(R_{i})\leq M, \qquad 
	\max_{1\leq i\leq h} |R_{i}(\xi)| 
	< M^{- \widehat{w}_{n}(\xi)+\epsilon }.
	\]
	Since $d\leq n$ and $k\leq 2n-2<2n-1$ we have $h\geq 2(k-n+1)$.
	Again we identify $R_{i}$ with its coefficient vector $\underline{x}_{i}\in\mathbb{Z}^{k+1}$ and
	write $L_{k,R_{i}}^{\ast}=L_{k,\underline{x}_{i}}$. Then for the 
	induced functions, Lemma~\ref{lemuren} and \eqref{eq:byeqab} 
	%Hence, as $\epsilon_{1}$ can be arbitrarily small, we get 
	%
	%\[
	%\overline{\psi}_{n,1}^{\ast} \leq \frac{n+1}{n} %\frac{1}{1+\widehat{w}_{n}(\xi)}- \frac{1}{n}+\epsilon_{1} = %\frac{n-\widehat{w}_{n}(\xi)}{n(1+\widehat{w}_{n}(\xi))}.
	%\]
	%
	gives rise to positions $Q^{\prime}=e^{q \prime}$ with
	\begin{align*}
	\psi_{k,h}^{\ast}(Q^{\prime}) &\leq \max_{1\leq i\leq h}  \frac{L_{k,R_{i}}^{\ast}(q^{\prime})}{q^{\prime}}
	\leq \Phi_{k,n}\left( \max_{i=1,2} \frac{L_{n,P_{i} }^{\ast}(q)}{q}\right) \\ &\leq  \Phi_{k,n}\left(\frac{n-\widehat{w}_{n}(\xi)}{n(1+\widehat{w}_{n}(\xi))}\right)+\epsilon_{2}
	= \frac{n(k+1)}{(n+1)k} \frac{n-\widehat{w}_{n}(\xi)}{n(1+\widehat{w}_{n}(\xi))}+ \frac{k-n}{k(n+1)}+\epsilon_{2}.
	\end{align*} 
	Since there are arbitrarily large such $Q^{\prime}$ and $\epsilon$ can be taken arbitrarily small
	\begin{equation}  \label{eq:froma}
	\underline{\psi}_{k,h}^{\ast} \leq \frac{n(k+1)}{(n+1)k} \frac{n-\widehat{w}_{n}(\xi)}{n(1+\widehat{w}_{n}(\xi))}+ \frac{k-n}{k(n+1)}.
	\end{equation} 
	Using \eqref{eq:jaja} and \eqref{eq:thatp} we can estimate
	\begin{equation} \label{eq:letztm}
	\underline{\psi}_{k,1} = -\overline{\psi}_{k,k+1}^{\ast} 
	\leq \frac{h}{k+1-h} \underline{\psi}_{k,h}^{\ast}, \qquad \text{if}\; k\leq 2n-2. 
	\end{equation}
	The condition on $k$ ensures $k+1-h>0$.
	Inserting the bound for $\underline{\psi}_{k,h}^{\ast}$
	from \eqref{eq:froma} and the worst case $h=2(k-n+1)$
    in \eqref{eq:letztm} and applying \eqref{eq:umrechnen}, we get \eqref{eq:jjj} after some calculation.
		
	For \eqref{eq:AA1}, we notice that if the degree of $P_{1}$ above is $d=n$ then
	we can proceed as in the proof of Theorem~\ref{th3} to get 
	its bound \eqref{eq:tarda}.
	Otherwise $d\leq n-1$ and thus now $h\geq 2k-2n+3$, and as soon
	as $k\leq 2n-3$ we can proceed as above to
	obtain the other bound when using $h=2k-2n+3$ in \eqref{eq:letztm}.

\section{Proof of the going-up Theorem~\ref{gutersatze0}} \label{7}

Similar ideas as for the mixed inequalities
are used to prove the estimates that contain
only simultaneous approximation exponents $\lambda_{N}(\xi)$. However,
roughly speaking, one more step of duality considerations between
simultaneous and linear form approximation is required here. 
We apply the same notational simplifications as in Section~\ref{6}.

\subsection{Proof of \eqref{eq:folgerunge0}}
	Let $\xi$ be a real number. It
	follows from \eqref{eq:umrechnen} that for any $\epsilon>0$
	there exist arbitrarily large parameters $Q$ such that
	\begin{equation} \label{eq:dort0}
	|\psi_{n,1}(Q)- \frac{1-n\lambda_{n}}{n(1+\lambda_{n})}|< \epsilon.
	\end{equation}
	Consider such large $Q$ fixed and let $q=\log Q$. 
	When we transition to the linear form problem,
	together with \eqref{eq:schranke} 
	we infer
	\begin{equation} \label{eq:itz0}
	\psi_{n,1}^{\ast}(Q)+\cdots+\psi_{n,n}^{\ast}(Q)\leq \frac{1-n\lambda_{n}}{n(1+\lambda_{n})}+\epsilon+O(q^{-1}).
	\end{equation}
	We also want to bound $\psi_{n,1}^{\ast}(Q)$ from above.
	We could estimate it by the
	right hand side of \eqref{eq:itz0} divided by $n$, 
	which would turn out to reprove
	Theorem~\ref{juppy},
	 but using the uniform
	exponent we find a better bound.
	
	From \eqref{eq:dort0} we
	obtain points $(q,L_{n,1}(q))$ with arbitrarily large
	$q$ and the property 
	\begin{equation}  \label{eq:yy}
	(-\alpha-\epsilon) q\leq L_{n,1}(q)\leq (-\alpha+\epsilon) q,
	\end{equation}
	where we have put
	\begin{equation} \label{eq:alpha0}
	-\alpha=\frac{1-n\lambda_{n}}{n(1+\lambda_{n})}.
	\end{equation}
	for simplicity. We can assume that $L_{n,1}$ has a local minimum
	at $q$. Then in some interval $[q-s,q]$ the function $L_{n,1}$
	decays with slope $-1$. The switch point $q-s$, where
	$L_{n,1}$ changes slope from $1/n$ to $-1$, is where it meets
	the second minimum function $L_{n,2}$. At $q-s$, again
	from \eqref{eq:umrechnen} we obtain
	\[
	L_{n,1}(q-s)=L_{n,2}(q-s)\leq \frac{1-n\widehat{\lambda}_{n}}{n(1+\widehat{\lambda}_{n})}(q-s)+ \epsilon_{1}q.
	\]
	%
	%We also can assume that $L_{n,1}(q-s)$ differs at most $o(1)$ as %$q\to\infty$ from
	%the right hand side by choosing appropriate $q$, so we also have 
	%the same lower bound up to $o(1)$. 
	Again let
	\[
	-\beta:=
	\frac{1-n\widehat{\lambda}_{n}}{n(1+\widehat{\lambda}_{n})}.
	\]
	Since $L_{n,1}$ decays with slope
	$-1$ in $[q-s,q]$, on the other hand by \eqref{eq:yy} we have
	\[
	L_{n,1}(q-s)= L_{n,1}(q)+s=(-\alpha+\delta) q+s,
	\]
	where $\delta\in (-\epsilon,\epsilon)$ is of small modulus.
	Equating the two expressions for $L_{n,1}(q-s)$, after some calculation 
	we get
	\[
	0<s\leq %q\cdot \frac{\alpha-\beta}{\beta-1}=
	q\cdot \frac{\lambda_{n}-\widehat{\lambda}_{n}}{1+\lambda_{n}}+ \epsilon_{2}q.
	\]
	As the second successive
	minimum has slope at most $1/n$ in
	$[q-s,q]$, inserting for $s$, at position $q$ we get
	\[
	L_{n,2}(q)\leq L_{n,2}(q-s)+\frac{1}{n}s\leq \frac{\lambda_{n}-(n+1)\widehat{\lambda}_{n}+1}{n(1+\lambda_{n})} q+\epsilon_{3}q.
	\]
	Let
	\begin{equation} \label{eq:gamma0} 
	-\gamma
	%\alpha-\frac{l+1}{l}\frac{t}{q}= 
	%\frac{l\lambda_{l}-1}{l(1+\lambda_{l})}-
	%\frac{l+1}{l}
	%\frac{\widehat{\lambda}_{l}-\lambda_{l}}{1+\lambda_{l}}
	:= \frac{\lambda_{n}-(n+1)\widehat{\lambda}_{n}+1}{n(1+\lambda_{n})}.
	\end{equation}

	Now again consider the dual linear form problem
	with respect to $(\xi,\xi^{2},\ldots,\xi^{n})$.
	Recall the notation $q=\log Q$ and
	$\psi_{n,j}^{\ast}(Q)=L_{n,j}^{\ast}(q)/q$.
	By Mahler's duality \eqref{eq:mahler}, for the last two successive minima 
	at position $q$ we obtain 
	\begin{equation}  \label{eq:p1}
	\psi_{n,n+1}^{\ast}(Q)=\frac{L_{n,n+1}^{\ast}(q)}{q}\geq -\frac{L_{n,1}(q)}{q}-O(q^{-1})\geq
	(\alpha-\epsilon) -O(q^{-1})%=\frac{l\lambda_{l}-1}{l(1+\lambda_{l})}-O(q^{-1}),
	\end{equation}
	and
	\begin{equation} \label{eq:p2}
	\psi_{n,n}^{\ast}(Q)=\frac{L_{n,n}^{\ast}(q)}{q}\geq -\frac{L_{n,2}(q)}{q}-O(q^{-1})=(\gamma-\epsilon_{3})-O(q^{-1}).
	\end{equation}
	Since the sum of all $n+1$ successive minima 
	functions $\psi_{n,j}^{\ast}$ at $Q$ is $O(q^{-1})$
	by \eqref{eq:rrr}, we have that
	\[
	\sum_{j=1}^{n-1} \psi_{n,j}^{\ast}(Q)\leq -\psi_{n,n}^{\ast}(Q)-\psi_{n,n+1}^{\ast}(Q)+O(q^{-1})\leq -\alpha-\gamma+\epsilon+\epsilon_{3}+O(q^{-1}).
	\]
	In particular
	\begin{equation} \label{eq:oi0}
	\psi_{n,1}^{\ast}(Q)\leq \frac{\sum_{j=1}^{n-1} \psi_{n,j}^{\ast}(Q)}{n-1}\leq 
	\frac{-\alpha-\gamma}{n-1}+\epsilon_{4}+O(q^{-1}).
	\end{equation}
	This is the desired bound for $\psi_{n,1}^{\ast}(Q)$.
	
	Now we transition to dimension $k\geq n$.
	Each of the pairs $(Q,\psi_{n,j}^{\ast}(Q))$ are induced
	by $L_{n,\underline{x}_{j}}^{\ast}$ as defined
	in \eqref{eq:defin} for some
	$\underline{x}_{j}=(x_{j,0},\ldots,x_{j,n})$. We identify
	each $\underline{x}_{j}$
	with the
	polynomial $P_{j}(T)=x_{j,0}+x_{j,1}T+\cdots+x_{j,n}T^{n}$ again. 
	Moreover $\mathscr{P}=\{P_{1},\ldots,P_{n}\}$ are linearly independent. Let $d$ be the degree of $P_{1}$, that
	is the largest index with $x_{1,d}\neq 0$. Clearly $1\leq d\leq n$.
	Starting from these polynomials we
	derive the ordered set of $k$ polynomials
	\[
	\mathscr{R}=\{R_{1},\ldots,R_{k}\}= \{ P_{1},T^{l-d+1}P_{1},\ldots,T^{k-d}P_{1},P_{2},P_{3},\ldots,P_{n}\}.
	\]
	Any $R_{i}$ has degree at most $k$. Furthermore
	it is easy to check that the linear independence of
	$\mathscr{P}$ implies that $\mathscr{R}$ is
	linearly independent as well, since starting with
	$\mathscr{P}$ and adding 
	one by one the 
	new polynomials in
	$\mathscr{R}\setminus \mathscr{P}=\{ T^{n-d+1}P_{1},\ldots,T^{k-d}P_{1}\}$,
	the dimension of the span increases in each step because
	the new polynomial has strictly larger degree than all the
	previous polynomials, and thus does not lie in their span.

	For simplicity now assume the typical case $d=n$, 
	otherwise the correspondence to Lemma~\ref{lemuren} 
	in following argument has to be slightly
	altered, and the remainder of the proof remains unaffected anyway.
	Then the first $k-n$ polynomials $R_{1},\ldots,R_{k-n}$ correspond to
	vectors $\underline{x}^{\prime}_{i}$ in \eqref{eq:ous2} 
	for $1\leq i\leq k-n$ in
	Lemma~\ref{lemuren} for $\underline{x}=\underline{x}_{1}$ 
	the coefficient vector of $P_{1}$, and similarly
	$R_{k-n+j}$ to $\underline{x}^{\prime}$ in \eqref{eq:ous} 
	for $\underline{x}=\underline{x}_{j}$
	the coefficient vector of $P_{j}$ as defined above, 
	for $1\leq j\leq n$ (so $\underline{x}_{1}$ appears $k-n+1$ 
	times in total).
	We may assume $\xi\in (0,1)$
	and apply Lemma~\ref{lemuren} to each $R_{i}$.
	With $Q^{\prime}=e^{q^{\prime}}$ for $q^{\prime}$ in
	\eqref{eq:os}, from the linear independence 
	of $\mathscr{R}$ we obtain
	\begin{align*}
	\psi_{k,j}^{\ast}(Q^{\prime})&\leq \Phi_{k,n}(\psi_{n,1}^{\ast}(Q))=
	(\psi_{n,1}^{\ast}(Q)-1)\frac{k+1}{k}\frac{n}{n+1}+1, 
	\qquad 1\leq j\leq k-n,   \\
	\psi_{k,k-n+j}^{\ast}(Q^{\prime})&\leq \Phi_{n,k}(\psi_{n,j}^{\ast}(Q))=  
	(\psi_{n,j}^{\ast}(Q)-1)\frac{k+1}{k}\frac{n}{n+1}+1, 
	\qquad 1\leq j\leq n.
	\end{align*}
	Summing over $j=1,2,\ldots,k$ we infer
	\[
	\sum_{j=1}^{k} \psi_{k,j}^{\ast}(Q^{\prime})\leq (k-n)A+B,
	\]
	where
	\[
	A=(\psi_{n,1}^{\ast}(Q)-1)\frac{(k+1)n}{k(n+1)}+1, \quad B=\sum_{j=1}^{n} \left[(\psi_{n,j}^{\ast}(Q)-1)\frac{(k+1)n}{k(n+1)}+1\right].
	\]
	This can be equivalently written
	\[
	\sum_{j=1}^{k} \psi_{k,j}^{\ast}(Q^{\prime})\leq 
	(k-n)\frac{(k+1)n}{k(n+1)}\psi_{n,1}^{\ast}(Q) 
	+ \frac{(k+1)n}{k(n+1)} \sum_{j=1}^{n} \psi_{n,j}^{\ast}(Q)
	 + \frac{k-n}{n+1}.
	\]
	Now we use the estimates \eqref{eq:itz0} and \eqref{eq:oi0} and
	inserting for $\alpha, \gamma$ from \eqref{eq:alpha0}, \eqref{eq:gamma0}
	after some calculation we end up at
	\[
	\sum_{j=1}^{k} \psi_{k,j}^{\ast}(Q^{\prime})
	\leq \frac{k(1-n)\lambda_{n}+(kn+n-k^{2}-k)
		\widehat{\lambda}_{n}+k^2-kn+k-1}{k(n-1)(\lambda_{n}+1)}
	+\epsilon_{5}.
	\]
	Together with \eqref{eq:schranke} this implies
	for large $Q$ we derive the estimate
	\[
	\psi_{k,1}(Q^{\prime})\leq \frac{k(1-n)\lambda_{n}+(kn+n-k^{2}-k)
		\widehat{\lambda}_{n}+k^2-kn+k-1}{k(n-1)(\lambda_{n}+1)}
	+\epsilon_{6}.
	\]
	Since there are arbitrarily large $Q$ and thus induced $Q^{\prime}$
	with this property, we derive
	\[
	\underline{\psi}_{k,1}\leq \frac{k(1-n)\lambda_{n}+(kn+n-k^{2}-k)
		\widehat{\lambda}_{n}+k^2-kn+k-1}{k(n-1)(\lambda_{n}+1)}+\epsilon_{6}.
	\]
	Inserting in \eqref{eq:umrechnen} we derive the desired
	estimate \eqref{eq:folgerunge0} after some calculation 
	and $\epsilon\to 0$. 

\begin{remark}  \label{rehhi}
	From \eqref{eq:oi0} when inserting for $\alpha, \gamma$
	in \eqref{eq:alpha0}, \eqref{eq:gamma0}
	and applying \eqref{eq:umrechnen2} we get a new proof of the inequality
	\[
	w_{k}(\xi) \geq \frac{(k-1)\lambda_{k}(\xi)+\widehat{\lambda}_{k}(\xi)+k-2}{1-\widehat{\lambda}_{k}(\xi)}, \qquad\qquad k\geq 2,
	\]
	already obtained by Bugeaud, Laurent~\cite{bl2010}, and
	with a different proof by Schmidt and Summerer~\cite{ssch}. 
	Again our proof of this estimate, as in~\cite{bl2010} and~\cite{ssch}, extends to the general case of $\mathbb{Q}$-linearly independent vectors $\{1, \xi_{1},\ldots,\xi_{k} \}$.
\end{remark}

The proof of \eqref{eq:wanndenn} is very similar, with a
slightly different strategy for estimation.

\subsection{Proof of \eqref{eq:wanndenn}}
	We proceed precisely as in the proof of \eqref{eq:folgerunge0} 
	up to \eqref{eq:oi0}.
	From \eqref{eq:p1}, \eqref{eq:p2} and
	since the sum of all $n+1$ successive minima 
	functions $\psi_{n,j}^{\ast}$ at $Q$ is $O(q^{-1})$
	by \eqref{eq:rrr}, we have that
	\begin{equation} \label{eq:uuu}
	\sum_{j=1}^{n-1} \psi_{n,j}^{\ast}(Q)\leq -\psi_{n,n}^{\ast}(Q)-\psi_{n,n+1}^{\ast}(Q)+O(q^{-1})\leq -\alpha-\gamma+\epsilon_{7}+O(q^{-1}).
	\end{equation}
	We will use this in place of \eqref{eq:itz0},
	and combine it again with \eqref{eq:oi0}.
	%
	%Moreover
	%\begin{equation} \label{eq:oii}
	%\varphi_{l,1}^{\ast}(Q)\leq -\frac{-\varphi_{l,l+1}^{\ast}(Q)}{l}\leq
	%-\frac{\alpha}{l}+O(q^{-1}).
	%\end{equation}
	%

	Now we transition to dimension $k\geq n$. 
	Let $\underline{x}_{1},\ldots,
	\underline{x}_{n-1}$ be the linearly independent integer vectors inducing
	$L_{n,j}^{\ast}(q)$
	(or equivalently $\psi_{n,j}^{\ast}(Q)$) for $1\leq j\leq n-1$
	as above. They correspond to polynomials
	$P_{j}(T)=x_{j,0}+x_{j,1}T+\cdots+x_{j,n}T^{n}$.
	Let 
	\[
	\mathscr{P}= \{ P_{1},\ldots,P_{n-1}\}.
	\]
	Let $d$ be the degree of $P_{1}$, that
	is the largest index with $x_{1,d}\neq 0$. 
	Clearly $1\leq d\leq n$. 
	Starting from these polynomials we
	derive the ordered set of $k-1$ polynomials
	\[
	\mathscr{R}=\{R_{1},\ldots,R_{k-1}\}= \{ P_{1},T^{n-d+1}P_{1},
	T^{n-d+2}P_{1}(T),\ldots,T^{k-d}P_{1},P_{2},P_{3},\ldots,P_{n-1}\}.
	\]
	Any $R_{i}$ has degree at most $k$. Furthermore
	it is easy to check that the linear independence of
	$\mathscr{P}$ implies that $\mathscr{R}$ is
	linearly independent as well. Indeed, starting with
	$\mathscr{P}$ and adding 
	one by one the 
	new polynomials in
	$\mathscr{R}\setminus \mathscr{P}=\{ T^{n-d+1}P_{1},\ldots,T^{k-d}P_{1}\}$,
	the dimension of the span increases in each step because
	the new polynomial has strictly larger degree than all the
	previous polynomials, and thus does not lie in their span.
	
	The polynomials $P_{j}\in\mathscr{P}$ give rise to
	points $\underline{x}^{\prime}_{1},\ldots,\underline{x}^{\prime}_{n-1}$ 
	as in \eqref{eq:ous} via embedding them into $\mathbb{Z}^{k+1}$.
	Write 
	$\psi_{k,P_{j}}^{\ast}(Q^{\prime})=
	\psi_{k,\underline{x}_{j}^{\prime}}^{\ast}(Q^{\prime})=
	L_{k,\underline{x}_{j}^{\prime}}^{\ast}(q^{\prime})/q^{\prime}$  
	with the functions $L_{k,\underline{x}_{j}^{\prime}}^{\ast}$ 
	as in \eqref{eq:defin} for the polynomial $P_{j}$ above, 
    and the corresponding notation for other polynomials.	
	With $\Phi_{k,n}$
	as in \eqref{eq:phink}, from Lemma~\ref{lemuren} 
	and \eqref{eq:uuu} we get some point $Q^{\prime}=e^{q^{\prime}}$
	where we have
	\begin{align*}
	\sum_{P\in \mathscr{P} } \psi_{k,P}^{\ast}(Q^{\prime})= \sum_{j=1}^{n-1}
	\psi_{k,P_{j}}^{\ast}(Q^{\prime})&=
	\sum_{j=1}^{n-1} \psi_{k,\underline{x}^{\prime}_{j}}^{\ast}(Q^{\prime})\\
	&= \sum_{j=1}^{n-1} \Phi_{k,n}(\psi_{n,j}^{\ast}(Q))\\
	&=
	\Phi_{k,n}(\sum_{j=1}^{n-1} \psi_{n,j}^{\ast}(Q))+(n-2)\frac{k-n}{k(n+1)} \\
	&\leq \Phi_{k,n}( -\alpha-\gamma)+\epsilon_{8}+(n-2)\frac{k-n}{k(n+1)}+O(q^{-1}).
	\end{align*}
	Hereby we used the fact that $\Phi_{k,n}$ are
	 affine functions with constant term $(k-n)/(k(n+1))$.
	Assume without loss of generality $\xi\in(0,1)$.
	Then, again by Lemma~\ref{lemuren}, 
	for the remaining $(k-1)-(n-1)=k-n$ polynomials 
	in $\mathscr{R}\setminus \mathscr{P}$ we obtain
	\[
	\sum_{R\in \mathscr{R}\setminus \mathscr{P}}
	\psi_{k,R}^{\ast}(Q^{\prime})\leq 
	(k-n) \Phi_{k,n}(\psi_{n,1}^{\ast}(Q)).
	\]
	The entire sum over $\mathscr{R}= \mathscr{P}\cup (\mathscr{R}\setminus \mathscr{P})$ is the sum of both 
	left hand sides above, thus by \eqref{eq:oi0} we infer
	\begin{align*}
	\sum_{R\in \mathscr{R}}\psi_{k,R}^{\ast}(Q^{\prime})&\leq  \Phi_{k,n}( -\alpha-\gamma)+(k-n) \Phi_{k,n}(\psi_{n,1}^{\ast}(Q))+(n-2)\frac{k-n}{k(n+1)}+\epsilon_{9}+O(q^{\prime -1})\\
	&\leq \Phi_{k,n}( -\alpha-\gamma)+(k-n) \Phi_{k,n}(-\frac{\alpha+\gamma}{n-1})+(n-2)\frac{k-n}{k(n+1)}+\epsilon_{10}+O(q^{\prime -1}).
	\end{align*}
	As $\mathscr{R}$ is a linearly 
	independent set of cardinality $k-1$ here, we may write this as
	\[
	\sum_{j=1}^{k-1} \psi_{k,j}^{\ast}(Q^{\prime})\leq 
	\sum_{R\in \mathscr{R}}\psi_{k,R}^{\ast}(Q^{\prime})\leq 
	\tau+\epsilon_{10}+O(q^{-1}), 
	\]
	where inserting in $\Phi_{k,n}$ we calculate
	\begin{align*}
	\tau:&=\Phi_{k,n}( -\alpha-\gamma)+(k-n) \Phi_{k,n}(-\frac{\alpha+\gamma}{n-1})+(n-2)\frac{k-n}{k(n+1)}\\
	%&=-(\alpha+\gamma)\cdot %\left(1+\frac{k-n}{n-1}\right)\frac{(k+1)n}{k(n+1)} +(1+k-n+n-2)\cdot %\frac{k-n}{k(n+1)}
	&=-(\alpha+\gamma)\frac{(k+1)(k-1)n}{(n+1)(n-1)k}+ \frac{(k-1)(k-n)}{k(n+1)}.
	\end{align*}
	Now for $\psi_{k,k+1}^{\ast}(Q^{\prime})$, which is the average
	slope of the last minimum function $L_{k,k+1}^{\ast}$ in
	$[0,q^{\prime}]$, by \eqref{eq:rrr}
	we obtain
	\[
	\psi_{k,k+1}^{\ast}(Q^{\prime})\geq 
	- \frac{\sum_{j=1}^{k-1} \psi_{k,j}^{\ast}(Q^{\prime})}{2}-O(q^{-1})\geq
	-\frac{\tau}{2}-\epsilon_{11}-O(q^{-1}).
	\]
	Again by Mahler's duality \eqref{eq:mahler}, for the first minimum
	of the simultaneous approximation problem in dimension
	$k$ at $q$ we get
	\begin{equation} \label{eq:gl}
	\psi_{k,1}(Q^{\prime})\leq -\psi_{k,k+1}^{\ast}(Q^{\prime})
	+ O(q^{\prime -1})\leq \frac{\tau}{2}+\epsilon_{11}+O(q^{\prime -1}).
	\end{equation}
	Using
	\begin{equation} \label{eq:ssb}
	\lambda_{k}(\xi)\geq \limsup_{Q^{\prime}\to\infty} \frac{1-k\psi_{k,1}(Q^{\prime})}{k+k\psi_{k,1}(Q^{\prime})}
	\end{equation}
	from \eqref{eq:umrechnen} again, 
	from \eqref{eq:gl} we get a lower bound for
	$\lambda_{k}(\xi)$ in terms of $\tau$, which in turn
	depends only on $\alpha,\gamma$.
	Inserting for $\alpha,\gamma$ from \eqref{eq:alpha0}, \eqref{eq:gamma0} for large enough $q\geq q_{0}(\epsilon)$ 
	as above,
	after a tidious calculation and rearrangement,
	we end up at
	\begin{equation}  \label{eq:ooo}
	\lambda_{k}\geq \frac{(n-1)\lambda_{n}+
		(k-1)\widehat{\lambda}_{n}+n-k}{(n-1)\lambda_{n}-(k-1)\widehat{\lambda}_{n}+n+k-2}-\epsilon_{12}.
	\end{equation}
	%
	%and in particular we used the identity
	%
	%\[
	%-B= (l+1)(l-1)+(l+1-2k)(k+1)-k(k-l)(l-1).
	%\]
	%.
	Since we can choose $\epsilon$ arbitrarily small,
	the bound becomes as in the theorem.

\section{Deduction of the results from Section~\ref{qli}} \label{pqli}

Let $N\geq 1$ be an integer. 
Assume $\{1,\zeta_{1},\ldots,\zeta_{N}\}$ is linearly independent
over $\mathbb{Q}$ and write $\underline{\zeta}=(\zeta_{1},\ldots,\zeta_{N})$.
Khintchine's transference principle~\cite{kredi} states
\begin{equation} \label{eq:khintchine}
\frac{w_{N}(\underline{\zeta})}{(N-1)w_{N}(\underline{\zeta})+N}\leq 
\lambda_{N}(\underline{\zeta})\leq \frac{w_{N}(\underline{\zeta})-N+1}{N}.
\end{equation}
We recall the refinements in terms of introducing uniform exponents 
\begin{equation} \label{eq:toroben}
\lambda_{N}(\underline{\zeta}) \geq \frac{(\widehat{w}_{N}(\underline{\zeta})-1)w_{N}(\underline{\zeta})}{((N-2)\widehat{w}_{N}(\underline{\zeta})+1)w_{N}(\underline{\zeta})+(N-1)\widehat{w}_{N}(\underline{\zeta})}
\end{equation}
and
\begin{equation} \label{eq:tur}
w_{N}(\underline{\zeta}) \geq \frac{(N-1)\lambda_{N}(\underline{\xi})+\widehat{\lambda}_{N}(\underline{\zeta})+N-2}{1-\widehat{\lambda}_{N}(\underline{\zeta})},
\end{equation}
already
quoted below Theorem~\ref{t2} and Remark~\ref{rehhi}, respectively.
Considering only uniform exponents, German~\cite{ogerman} showed
\begin{equation} \label{eq:ogerman}
\frac{\widehat{w}_{N}(\underline{\zeta})-1}{(N-1)\widehat{w}_{N}(\underline{\zeta})}\leq
\widehat{\lambda}_{N}(\underline{\zeta})\leq \frac{\widehat{w}_{N}(\underline{\zeta})-N+1}{\widehat{w}_{N}(\underline{\zeta})}.
\end{equation}
Estimates \eqref{eq:khintchine}, \eqref{eq:ogerman}
are best possible, and \eqref{eq:toroben}, \eqref{eq:tur} at least
for $N=2$ as well~\cite{bl2010},\cite{laurent}. In the Remark on page 80 in~\cite{ss}, a short proof
of \eqref{eq:khintchine} that only uses parametric
geometry of numbers is given. It resembles our proofs from Sections~\ref{6},~\ref{7} that implicitly
recover the refined estimates \eqref{eq:toroben}, \eqref{eq:tur}. 
For an alternative proof of \eqref{eq:ogerman}
and its optimality based on parametric geometry of numbers, see~\cite{schms}.
It is worth noting that \eqref{eq:ogerman} is stronger than the analogue of \eqref{eq:khintchine}
obtained from replacing ordinary by uniform exponents.
In the proofs, we apply above estimates to finite dimensional 
projections of $\underline{\xi}\in\mathbb{R}^{\mathbb{N} }$.

\begin{proof}[Proof of Theorem~\ref{allgemein}]
	We start
	with the inequality \eqref{eq:tur} for $N=n$ and
	$\underline{\zeta}=\underline{\xi}_{n}$.
	On the other hand it is easy to see that for any 
	$\underline{\xi}\in \mathbb{R}^{\mathbb{N} }$ and $k\geq n$
	\begin{equation} \label{eq:nruno}
	w_{k}(\underline{\xi}_{k})\geq w_{n}(\underline{\xi}_{n}),
	\end{equation}
	since for any vector $\underline{x}=(x_{0},\ldots,x_{n})$ as in the definition
	of $w_{n}$ taking $\underline{x}^{\prime}=(x_{0},\ldots,x_{n},0,\ldots,0)\in\mathbb{Z}^{k+1}$ yields $\Vert\underline{x}\Vert_{\infty}=\Vert\underline{x}^{\prime}\Vert_{\infty}$ and $\scp{\underline{x},\underline{\xi}_{n}}_{n}=\scp{\underline{x}^{\prime},\underline{\xi}_{k} }_{k}$.
	Combining yields
	\begin{equation} \label{eq:letzt}
	w_{k}(\underline{\xi}_{k})\geq \frac{(n-1)\lambda_{n}(\underline{\xi}_{n})+\widehat{\lambda}_{n}(\underline{\xi}_{n})+n-2}{1-\widehat{\lambda}_{n}(\underline{\xi}_{n})}=B,
	\end{equation}
	with $B$ as defined in the theorem.
	Now apply the left inequality from \eqref{eq:khintchine} with $N=k$,
	$\underline{\zeta}=\underline{\xi}_{k}$  
	to \eqref{eq:letzt} to obtain the bound \eqref{eq:bo1} after
	a short calculation. 
	
	For \eqref{eq:bo2}, we notice that \eqref{eq:schwachmaten} 
	combined with the right estimate in \eqref{eq:ogerman} for $N=n$ 
	and $\underline{\zeta}=\underline{\xi}_{n}$ yields
	\[
	\widehat{w}_{k}(\underline{\xi}_{k}) \geq
	\frac{k-1}{1-\frac{\widehat{\lambda}_{n}(\underline{\xi}_{n})+n-2}{(n-1)(k-1)}}= \frac{(n-1)(k-1)^2}{nk-k-2n+3-\widehat{\lambda}_{n}(\underline{\xi}_{n})}=A,
	\]
	again with $A$ as defined in the theorem. Inserting this and \eqref{eq:letzt} in \eqref{eq:toroben} 
	with $N=k$ and $\underline{\zeta}=\underline{\xi}_{k}$ yields
	\eqref{eq:bo2}.
	\end{proof}

Starting with \eqref{eq:khintchine} in the proof, instead
of \eqref{eq:bo1}
we would have directly obtained Theorem~\ref{gutsatz}.
The proof of Theorem~\ref{schlechtsatz} relies solely on the inequalities in \eqref{eq:ogerman}. 

\begin{proof}[Proof of Theorem~\ref{schlechtsatz}]
Similar to \eqref{eq:nruno} we have 
\[
\widehat{w}_{k}(\underline{\xi}_{k})\geq \widehat{w}_{n}(\underline{\xi}_{n}).
\]
Together with \eqref{eq:ogerman} for $N=n$ and
$\underline{\zeta}=\underline{\xi}_{n}$ we infer
\[
\widehat{w}_{k}(\underline{\xi}_{k}) 
\geq \widehat{w}_{n}(\underline{\xi}_{n}) \geq \frac{n-1}{1-\widehat{\lambda}_{n}(\underline{\xi}_{n})}.
\]
Inserting in the left inequality of \eqref{eq:ogerman} with $N=k$ 
 and $\underline{\zeta}=\underline{\xi}_{k}$ yields the claim.
\end{proof}

\begin{proof}[Proof of Theorem~\ref{consid}]
	We combine 
	\[
	w_{k}(\underline{\xi}_{k})\geq w_{n}(\underline{\xi}_{n}), \qquad
	\widehat{w}_{k}(\underline{\xi}_{k})\geq \widehat{w}_{n}(\underline{\xi}_{n}),
	\]
	with \eqref{eq:toroben} for $N=k$ and
	$\underline{\zeta}=\underline{\xi}_{k}$.
	\end{proof}

\vspace{0.5cm}

{\em The author thanks Yann Bugeaud for fruitful discussions
that helped to improve the paper! The author further thanks the referee
for pointing out several small inaccuracies}

%Since $m\geq \lceil\widehat{w}_{n}(\xi)\rceil-n+1$ we have that
%$h=2k+2-m-n \geq 2k-2n+2$ if $k\leq \lceil\widehat{w}_{n}(\xi)\rceil$.

\end{document}